\documentclass{amsart}
\usepackage{amsmath,amssymb,amsthm,amsfonts,amscd,mathrsfs,tikz-cd}
\usepackage[all,cmtip]{xy}
\usepackage{graphicx}
\usepackage{subcaption}

\theoremstyle{plain}

\newtheorem{theorem}{Theorem}[section]
\newtheorem{proposition}[theorem]{Proposition}
\newtheorem{lemma}[theorem]{Lemma}
\newtheorem{corollary}[theorem]{Corollary}

\newtheorem{problem}[theorem]{Problem}
\newtheorem{question}[theorem]{Question}
\newtheorem*{proposition*}{Proposition}

\theoremstyle{definition}
\newtheorem{definition}[theorem]{Definition}
\newtheorem{example}[theorem]{Example}
\newtheorem{examples}[theorem]{Examples}

\theoremstyle{remark}
\newtheorem{remark}[theorem]{Remark}

\newtheorem{discussion}[theorem]{Discussion}

\newcommand{\secref}[1]{Section~\ref{#1}}
\newcommand{\thmref}[1]{Theorem~\ref{#1}}
\newcommand{\propref}[1]{Proposition~\ref{#1}}
\newcommand{\lemref}[1]{Lemma~\ref{#1}}
\newcommand{\corref}[1]{Corollary~\ref{#1}}

\newcommand{\exref}[1]{Example~\ref{#1}}
\newcommand{\queref}[1]{Question~\ref{#1}}
\newcommand{\remref}[1]{Remark~\ref{#1}}

\newcommand{\defref}[1]{Definition~\ref{#1}}
\newcommand{\figref}[1]{Figure~\ref{#1}}
\newcommand{\disref}[1]{Discussion~\ref{#1}}

\numberwithin{subsection}{theorem}
\def\map{\mathrm{map}}

\def\cat{{\sf{cat}}}
\def\dcat{{\sf{d\text{-}cat}}}

\newcommand{\be}{\begin{enumerate}}
\newcommand{\ee}{\end{enumerate}}
\newcommand{\R}{\mathbb{R}}
\newcommand{\Z}{\mathbb{Z}}
\newcommand{\N}{\mathbb{N}}
\newcommand{\C}{\mathbb{C}}

\begin{document}

\title[Homotopy Theory in  Digital Topology]{Homotopy Theory in  Digital Topology}

\author{Gregory Lupton}
\author{John Oprea}
\author{Nicholas A. Scoville}

\address{Department of Mathematics, Cleveland State University, Cleveland OH 44115 U.S.A.}

\email{g.lupton@csuohio.edu}
\email{j.oprea@csuohio.edu}

\address{Department of Mathematics and Computer Science, Ursinus College, Collegeville PA 19426 U.S.A.}

\email{nscoville@ursinus.edu}

\date{\today}


\begin{abstract} Digital topology is part of the ongoing endeavour to understand and analyze digitized images. With a view to supporting this endeavour, many notions from algebraic topology have been introduced into the setting of digital topology.  But some of the most basic notions from homotopy theory remain largely absent from the digital topology literature. We embark on a development of homotopy theory in digital topology, and define such fundamental notions as function spaces, path spaces, and cofibrations in this setting.  We establish digital analogues of basic homotopy-theoretic properties such as the homotopy extension property for cofibrations, and the homotopy lifting property for certain evaluation maps that correspond to path fibrations in the topological setting.  We indicate that some depth may be achieved by using these homotopy-theoretic notions to give a preliminary treatment of Lusternik-Schnirelmann category in the digital topology setting.  This topic provides a connection between digital topology and critical points of functions on manifolds, as well as other topics from topological dynamics.
\end{abstract}

\thanks{This work was partially supported by grants from the Simons Foundation: (\#209575 to Gregory Lupton
and \#244393 to John Oprea).}

\maketitle

\section{Introduction}

In digital topology, the basic object of interest  is a  \emph{digital image}: a finite set of integer lattice points in an ambient Euclidean space with a suitable adjacency relation between points.  This is an abstraction of an actual digital image which consists of  pixels (in the plane, or higher dimensional analogues of such).  There is an extensive literature with many results that apply topological ideas in this setting (e.g. \cite{Ro86, K-R-R92, Bo99, Evako2006}).   Many of these results are obtained by importing key topological concepts from the ordinary topological setting into the digital setting, which is more discrete or combinatorial, rather than topological in nature.  Concepts from point-set topology such as continuity, the Jordan curve theorem, arc connectedness, boundary, closure, and nowhere dense all have digital analogues \cite{Kong91}. Several attempts have been made to do algebraic topology  and in particular homotopy theory in the digital setting (e.g.~\cite{GR05, GLU14, MVK14}). But the combination of homotopy theory and digital topology is not yet really mature and most of the literature involves fairly elementary ingredients from algebraic topology, such as the fundamental group or (low-dimensional) homology groups (see, e.g., \cite{Bo99} or \cite{GLU14} for references).  Furthermore, in many instances the notions that have been  established are quite  restrictive with limited applicability.  The result has been that, so far, very little depth has been achieved by combining algebraic and digital topology. 

 In contrast to this existing literature, we seek to  build a more general ``digital homotopy theory" that brings the full strength of homotopy theory to the digital setting.  We use less rigid constructions, with a view towards broad applicability and greater depth of development.  We begin this project here:  we establish general
constructions such as mapping spaces, path spaces,  cofibrations, and certain path fibrations in the digital topology setting, and thereby bring some of the more sophisticated tools and methods from homotopy theory to bear on this topic.  Our development here is deep enough to allow for a preliminary discussion of Lusternik-Schnirelmann category in the digital topology setting, for instance.   Other contributions to the project are given in \cite{LOS19c, LOS19b}.   
At the end of the paper, we indicate how future work will continue to develop a more robust and fuller \emph{digital homotopy theory}.

We now discuss a sample of the literature in this area. Several of  the articles we mention here and elsewhere in this paper give references to  other recent work in the field.   The article \cite{K-R-R92} contains a basic introduction to digital topology and some common themes of the subject.   Various notions of continuous functions and their ramifications for such concepts as homeomorphisms, retracts, and homotopy equivalence are discussed in \cite{Han10}.      The fundamental group of a digital image is discussed in \cite{Kong89} and \cite{Bo99}, including its relation to products \cite{BK12} and Euler characteristic \cite{BS16}. Furthermore, first attempts at higher homotopy groups are discussed in \cite{ADFQ03} and \cite{MVK14}.  Covering spaces are studied in \cite{Ha05, Bo06, B-K10}.  General properties of homotopy and homotopy equivalence are investigated in \cite{Bo05, B-S17, B-S18, H-M-P-S15}.  A notion of fibration in the digital setting is given in \cite{Eg-Ka17}.   But a major drawback with many of these papers, from our point of view, is that the notions established tend to be very rigid.  A typical example of this issue is provided by \cite{Eg-Ka17}, which defines a fibration in the digital setting by directly translating the topological definition into the digital setting.  But then it is difficult to display an example of a fibration with this definition\footnote{In that paper, there is an error in the discussion of Section 4, which invalidates the example given there.  This leaves a constant map as the only example of a fibration presented.}, and no developments flow from the notion introduced.   Similarly, in \cite{Bo-Ve18, KaIs18} definitions of Lusternik-Schnirelmann category and topological complexity, each of which is a numerical homotopy invariant, are translated directly from the topological to the digital setting, again with the result that the digital versions of these invariants (as defined  there) are too rigid to allow for much development.

In contrast to this tendency in the literature of directly translating topological notions into the digital setting, we have found that  the essential notion of \emph{subdivision} should be used to develop less rigid notions better suited to homotopy theory.   For example, in \secref{sec: cofibrations} we develop a notion of cofibration that incorporates subdivisions in a crucial way.  There we establish basic examples that display a form of the homotopy extension property: a homotopy may be extended after allowing for suitable subdivisions.  This is a recurrent theme in our development.  We find that, to develop a less rigid theory (i.e., one with interesting examples), one should allow for suitable subdivisions in the definitions and constructions desired.  This philosophy is on display throughout.  We follow it, for instance, in our versions of the following: homotopy extension and lifting properties (\defref{def: digital cofibration}, \corref{cor: path fibration}); contractibility of one digital image in another (\defref{def: U contractible in X}).  We make some further comments along these lines in \secref{sec: future}.   In addition to the results of this paper, other aspects of our broader digital homotopy theory program are represented in the papers \cite{LOS19c, LOS19b}.  In \cite{LOS19b}, we establish crucial results about the behaviour of maps with respect to subdivision. In \cite{LOS19c} we give a treatment of a fundamental group in the digital setting in which subdivision plays a prominent role.  This paper and \cite{LOS19c, LOS19b} complement each other within our broad digital homotopy theory program.  However, this paper is independent of  \cite{LOS19c, LOS19b} with one exception: One item of \secref{sec: cat and TC} (\thmref{thm: 2D dcat well-def}) uses a result from \cite{LOS19b} in its proof.

A more general notion than that of a digital image that has also appeared in the literature is that of a \emph{tolerance space} (see \cite{So86, Peters12}).  The same notion is called a \emph{fuzzy space} in \cite{Po71}, which refers to earlier work on this topic by Zeeman and Poincar{\'e}.  In these references, and especially in \cite{Po71}, many basic ideas from algebraic topology are mentioned in the tolerance space setting.  One defines a tolerance space as a set with an adjacency relation (i.e., a reflexive, symmetric, but generally not transitive binary relation).  The notion is equivalent to that of a simple graph, with edges corresponding to adjacencies (except for self-adjacencies).  But one does not assume an embedding into an integral lattice.  In fact, any (finite) tolerance space---or simple graph---may be embedded as a digital image in some (perhaps high-dimensional) $\Z^n$.  But there is no canonical way of doing so.  We keep our focus on the digital setting, although some of the notions we use here could be developed in the tolerance space setting.  But we note  that there is not really a good notion of subdivision in the tolerance space setting, and so, in so far as notions from algebraic topology have been developed in that setting, we see the same type of rigidity mentioned above that does not seem well-suited to homotopy theory.

We briefly summarize the content and organization of the paper as follows.  \secref{sec: basics} serves as a brief introduction to digital topology, and at the same time sets some basic conventions and notation. Perhaps the main idea reviewed here is that of subdivision, in \defref{def: subdivision}.  We continue to introduce basic notions in \secref{sec: homotopy}, although here the notions of function space and  path space are new to digital topology.  The notion of homotopy we use is the expected one, obtained by translating the topological definition into the digital setting.  Whereas homotopy in this form  ((A) of \defref{def: Left-Right homotopy}) has appeared in the literature, we note that  the particular adjacency relation we use on the product leads to consequences that differ from some in the literature (see \remref{rem: graph product homotopy}). Furthermore, an exponential correspondence also allows us to treat homotopy from a path object point of view,  which is a new way of treating homotopy in the digital literature (\defref{def: digital homotopy (bis)}).
In \secref{sec: cofibrations} we introduce a notion of cofibration into  the  digital setting.  A key point to note here is that we avoid translating the topological definition directly into the digital setting.  Rather, the notion we give incorporates subdivision in a crucial way.  Doing so allows us to establish such basic examples as the inclusion of one or both endpoints into an interval as a cofibration (\thmref{thm: well-pointed interval}  and \thmref{thm: endpoint cofibration}).  In  \secref{sec: fibrations} we add significant depth to the development by building on these results.     Using a digital version of a theorem that, in the topological setting, translates cofibrations to fibrations (\thmref{thm: Digital Borsuk}), we establish that certain evaluation maps have an adapted form of the homotopy lifting property, namely, that they are path \emph{fibrations} in a certain sense (\corref{cor: path fibration}, \corref{cor: TC fibration}, \corref{cor: based path fibration}).   In \secref{sec: Diamond Winding Number} we consider loops in a particular digital image and develop an invariant in this context that closely resembles the winding number from the ordinary topological or complex analytical setting.   Our development continues to expand in \secref{sec: cat and TC}, where we offer a short treatment of Lusternik-Schnirelmann category.  This is a numerical invariant that, whilst well-known in the ordinary topological setting, has not been developed greatly in the digital setting and which could prove useful for feature recognition, for instance. We establish some basic facts about this invariant and calculate its value for a digital image that may be considered the prototype of a digital circle (\propref{prop: cat of D}).  This calculation relies on the results of \secref{sec: Diamond Winding Number}.  In \corref{cor: d-cat = d-secat},  we apply the results of  \secref{sec: fibrations} to characterize this numerical invariant in terms of local sections of a certain path fibration.  This is a noteworthy result, in terms of our ``digital homotopy theory agenda," as it illustrates the possibility of establishing digital versions of constructions and results from ordinary homotopy theory that have greater depth of development  than any that have previously appeared in the digital topology  literature.         In the final \secref{sec: future} we indicate some questions and directions for future work.

\section{Basic Notions: Adjacency, Continuity, Products, Subdivision}\label{sec: basics}

The topics in this section are standard in digital topology and appear frequently in the literature.  We include them here as a convenience for the reader, to establish notation, and to emphasize the particular ingredients that we will use in the sequel.

In this paper, a \emph{digital image} $X$ means a finite subset $X \subseteq \Z^n$ of the integral lattice in some $n$-dimensional Euclidean space, together with
a particular  adjacency relation inherited from that of $\Z^n$.  Namely, two points $x = (x_1, \ldots, x_n) \in \Z^n$ and  $y = (y_1, \ldots, y_n) \in \Z^n$  are adjacent if their 
coordinates satisfy $|x_i-y_i| \leq 1$ for each $i = 1, \dots, n$.

\begin{remark}
In the literature, it is common to allow for various choices of adjacency.  For example, a planar digital image is a subset of $\Z^2$ with either ``$4$-adjacency" or ``$8$-adjacency" (see, e.g.~Section 2 of \cite{Bo99}).  However, in this paper, we always assume (a subset of) $\Z^n$ has the highest degree of adjacency possible ($8$-adjacency in $\Z^2$, $26$-adjacency in $\Z^3$, etc.).   In fact, there is a philosophical reason for our fixed choice of adjacency relation:  It is effectively forced on us by the considerations of  \defref{def: products} and \exref{ex: diagonal} below.
\end{remark}

If $x, y \in X \subseteq \Z^n$, we write $x \sim_X y$ to denote that $x$ and $y$ are adjacent. For digital images $X\subseteq \Z^n$ and $Y \subseteq \Z^m$, a  function $f\colon X \to Y$  is \emph{continuous} if $f(x) \sim_Y f(y)$ whenever $x \sim_Xy$.
By a \emph{map} of digital images, we mean a continuous function.  Occasionally, we may encounter a non-continuous function of digital images.  But, mostly, we deal with maps---continuous functions---of digital images.

An \emph{isomorphism} of digital images is a continuous bijection $f \colon X \to Y$ that admits a continuous inverse $g \colon Y \to X$, so that we have $f\circ g = \mathrm{id}_Y$ and $g\circ f = \mathrm{id}_X$ (such a $g$ is necessarily bijective).  If $f \colon X \to Y$ is an isomorphism, then we say that $X$ and $Y$ are isomorphic digital images, and write $X \cong Y$.

\begin{examples}\label{ex:basic digital images}
We use the notation $I_N$ for the \emph{digital interval of length} $N$, namely $I_N \subseteq \Z$ consists of the integers from $0$ to $N$ in $\Z$, and consecutive
integers are adjacent. Thus, we have $I_1 = \{0, 1\}$, $I_2 = \{0, 1, 2\}$, and so-on.  As an example in $\Z^2$, consider $D = \{ (1, 0), (0, 1), (-1, 0), (0, -1) \}$, which may be viewed as a digital circle (see \figref{fig:D & C}).  Note that pairs of vertices all of whose coordinates differ by $1$, such as  $(1, 0)$ and  $(0, 1)$ here, are adjacent according to our definition.  Otherwise, $D$ would be disconnected.   This example is \emph{the Diamond}  and we will establish several facts about it in the paper, starting with  \propref{prop: diamond non-contractible}.
\begin{figure}[h!]
\centering
   \begin{subfigure}{0.4\linewidth} \centering
    \includegraphics[trim=225 510 250 80,clip,width=0.5\textwidth]{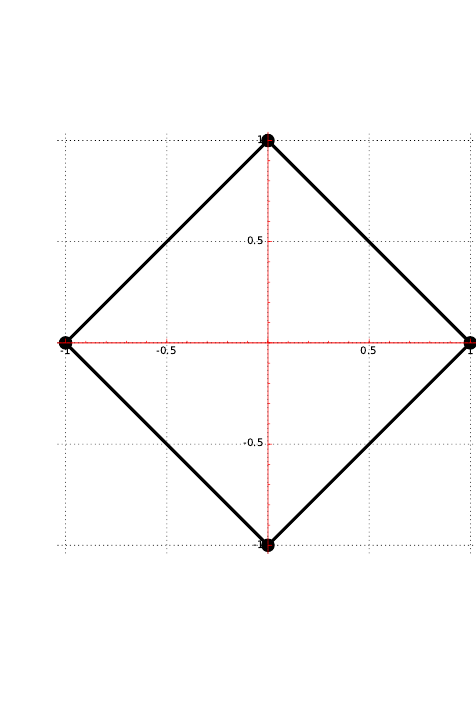}
     \caption{$D$: The Diamond}\label{fig:Circle D}
   \end{subfigure}
   \begin{subfigure}{0.59\linewidth} \centering
    \includegraphics[trim=180 445 195 120,clip,width=0.5\textwidth]{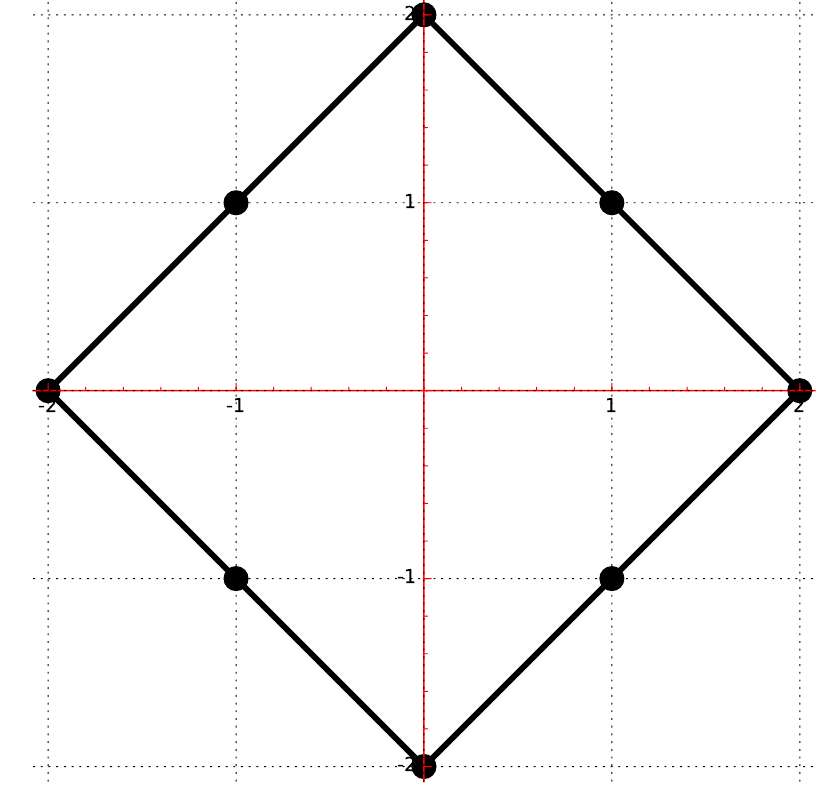}
     \caption{$C$: A larger digital circle}\label{fig:Circle C}
   \end{subfigure}
\caption{Two digital circles.} \label{fig:D & C}
\end{figure}
In \figref{fig:D & C} we have included the axes (in red) and also indicated adjacencies in the style of a graph.  Note, though, that we have no choice as to which points are adjacent: this is determined by position, or coordinates, and we do not choose to add or remove edges here.
 As an example in  $\Z^3$, we have $S = \{ (1, 0, 0), (0, 1, 0), (-1, 0, 0), (0, -1, 0), (0, 0, 1), (0, 0, -1) \}$ (the vertices of an octahedron, with adjacencies corresponding to the edges of the octahedron).  This may be viewed as a digital $2$-sphere, and the pattern emerging here may be continued to a digital $n$-sphere in $\Z^{n+1}$ with $2n+2$ vertices. 
 
 The map $f\colon I_2 \to I_1$ given by $f(0) = 0$, $f(1) = 0$, and $f(2) = 1$ is continuous, but the function  $g\colon I_1 \to I_2$ given by $g(0) = 0$, $g(1) = 2$ is not: we cannot ``stretch" an interval to a longer one.  Likewise, suppose we enlarge $D$ to the bigger digital circle   $C = \{ (2, 0), (1, 1), (0, 2), (-1, 1), (-2, 0), (-1, -1), (0, -2), (1, -1) \}$ (see \figref{fig:D & C}).  Then the only maps $D \to C$ will be ``homotopically trivial" maps---in a sense we will define later.  We cannot ``wrap" a smaller circle around a larger one.
\end{examples}

Because we want to use constructions such as the diagonal map as well as other maps into or out of products, we need to be clear about the adjacency relation in a product.

\begin{definition}[digital products]\label{def: products}
The \emph{product} of digital images $X$ and $Y$ is the Cartesian product of sets $X \times Y$ with the adjacency relation $(x, y) \sim_{X \times Y} (x', y')$ when $x\sim_X x'$  and $y \sim_Y y'$.
\end{definition}

In fact,  this is tantamount to our assumption that $\Z^n$, and any digital image in it, has the highest degree of adjacency possible, with the isomorphisms $\Z^n \cong \Z^r \times \Z^{n-r}$ for $r = 1, \ldots, n-1$.  Note that some authors in the literature use a different adjacency relation on the product: the \emph{graph product}, whereby $(x, y)$ is adjacent to $(x', y')$ if $x = x'$ and $y\sim_Y y'$, or $x \sim_X x'$ and $y = y'$.  The notion we use is sometimes called the  \emph{strong product}, in a graph theory setting.  Our definition of (adjacency on) the product means that it is the categorical product, in the category of (finite) digital images and digitally continuous maps.

\begin{example}\label{ex: diagonal}
Suppose we have $X = I_1 \subseteq \Z$.  Then the \emph{diagonal map} $\Delta\colon I_1 \to I_1 \times I_1 \subseteq \Z^2$ is given by $\Delta(x) = (x, x)$.  Since $0 \sim_X 1$, we need $(0,0) \sim_{X \times X} (1, 1)$ if the diagonal map is to be continuous, which of course we do have with our conventions.
\end{example}

\begin{definition}\label{def: map product}
Given maps of digital images $f_i \colon X_i \to Y_i$ for $i = 1, 2$, we define their product in the usual way as 
$$f_1 \times  f_2 \colon X_1 \times X_2 \to Y_1 \times Y_2$$
with $(f_1 \times f_2) (x_1,  x_2) = \big(f_1(x_1), f_2(x_2) \big)$.  The product of maps is a continuous map, as follows easily from the definitions.  
\end{definition}

Of course, the product of digital images and maps of digital images may be extended to any (finite) number of factors.

The notion of \emph{subdivision of a digital image} plays an important role in many of our definitions and constructions.

\begin{definition}\label{def: subdivision}
Suppose that $X$ is a digital image in $\Z^n$.  For each $k \geq 2$, we have  a \emph{$k$-fold subdivision of $X$}, which is an auxiliary (to $X$) digital image in $\Z^n$ denoted by $S(X, k)$, and also a \emph{canonical map} or projection
$$\rho_k\colon  S(X, k) \to X$$
that is continuous in the digital sense. This goes as follows.   For a real number $x$, denote by $\lfloor x \rfloor$ the greatest integer less-than-or-equal-to $x$.  First, make
the $\Z[1/k]$-lattice in $\R^n$, namely, those points with coordinates each of which is $z/k$ for some integer $z$, and then set
$$S'(X, k) = \left\{ (x_1, \ldots, x_n) \in \left(\Z\left[\frac{1}{k}\right]\right)^n  \mid ( \lfloor x_1 \rfloor, \ldots, \lfloor x_n \rfloor) \in X \right\}.$$
Then set
$$ S(X,k) = \left\{ (kx_1, \ldots, kx_n) \in \Z^n  \mid (x_1, \ldots, x_n) \in S'(X, k) \right\}.$$
The map $\rho_k$ is given by $\rho_k\big( (y_1, \ldots, y_n) \big) = ( \lfloor y_1/k \rfloor, \ldots, \lfloor y_n/k \rfloor)$, and one checks that this map is continuous.
For $x \in X$ an individual point, we write $S(x, k) \subseteq S(X, k)$ for the points $y \in S(X, k)$ that satisfy $\rho_k(y) = x$.  If $x = (x_1, \ldots, x_n)$ is a point in an $n$-dimensional digital image, then we may describe this set in general as
\begin{equation}\label{eq: point subdn}
S(x, k) = \{ (kx_1 + r_1, \ldots, kx_n + r_n) | 0 \leq r_i \leq k-1 \}.
\end{equation}
That is, for each $x \in X$, $S(x, k)$ is an $n$-dimensional cubical lattice in $\Z^n$ with each side of the cubical lattice containing $k$ points.  Notice that  the result of subdivision therefore depends on the ambient space of the digital image.
\end{definition}

In \cite{LOS19b}, we give a number of illustrative examples of subdivision.
Subdivision behaves well with respect to products.  For any digital images $X \subseteq \Z^m$ and $Y \subseteq \Z^n$ and any $k \geq 2$ we have an obvious isomorphism
$$S(X \times Y, k) \cong S(X, k) \times S(Y, k)$$
and, furthermore, the standard projection $\rho_k\colon S(X \times Y, k) \to X \times Y$ may be identified with the product of the standard projections on $X$ and $Y$, thus:
$$\rho_k = \rho_k \times \rho_k\colon  S(X, k) \times S(Y, k) \to X \times Y.$$
Note also that we may iterate subdivision.  It is straightforward to check that, for any $k, l$, we have an isomorphism of digital images 
$S\big( S(X, k), l\big) \cong S(X, kl)$.

By an \emph{inclusion of digital images} $j \colon A \to X \subseteq \Z^n$ we mean that $A$ is a subset of $X$ (the coordinates of a point of $A$ remain the same under inclusion into $X$).
It is easy to see that, given an \emph{inclusion} of digital images (of the same dimension) $j \colon A \to X \subseteq \Z^n$, we have an obvious corresponding continuous  inclusion of subdivisions $S(j, k)\colon S(A, k) \to S(X, k)$ such that the diagram
$$\xymatrix{
S(A, k) \ar[d]_{\rho_k} \ar[r]^{S(j, k)} & S(X, k) \ar[d]^{\rho_k}\\
A \ar[r]_{j} & X}$$
commutes.  We say that the map  $S(j, k)$ \emph{covers} the map $j$.   Indeed, we may give an explicit formula as follows.  For each point  $a \in A$, write $a = (a_1, \ldots, a_n)$. Also, write $t = (t_1, \ldots, t_n)$, with $0 \leq t_1, \ldots, t_n \leq k-1$, for a typical point $t$ in the cubical $k \times k \times \cdots \times k$ lattice $(I_{k-1})^n\subseteq \Z^n$.  Then the points of $S(a, k) \subseteq S(A, k)$ may be written as
$$S(a, k) = \{ k \,a + t \mid t \in (I_{k-1})^n \} = \{ (ka_1 + t_1, \ldots, ka_n+t_n) \mid 0 \leq t_1, \ldots, t_n \leq k-1 \},$$
with $\rho_k( k\, a + t) = a$ for all $t \in (I_{k-1})^n$.  Here, the scalar multiple $k\,a$ and the  sum $k\,a + t$ denote coordinate-wise (vector) scalar multiplication and addition in $\Z^n$.   Then $S(j, k)\colon S(A, k) \to S(X, k)$ may be written as
\begin{equation}\label{eq: canonical subdn map}
S(j, k)\big(  k\,a + t  \big) =   k\,j(a) + t,
\end{equation}
where $j(a) = (a_1, \ldots, a_n) \in X$.  It is easy to confirm that this gives a (continuous) map.
   We will make use of these induced maps of subdivisions in our development.

Note, however, that a general map $f \colon X \to Y$ may not induce a map of subdivisions, at least not in an obvious, canonical way.
In \cite{LOS19b}, we give a  full discussion of subdivision of a map.

Note also that, in general, we do not have a (continuous) right inverse to the projection $\rho_k\colon  S(X, k) \to X$.  There are a small number of exceptions to this general rule, and we make use of them in our development later.

\section{Function Spaces, Path Spaces, Homotopy}\label{sec: homotopy}

In this section, we introduce several topics that have not been studied previously in digital topology.
In homotopy theory, function spaces play a principal role.  For example, spaces of paths or loops are ubiquitous.  Furthermore the exponential correspondence, which (under mild hypotheses) gives a homeomorphism
$$\map(X \times Y, Z) \equiv \map\big(X, \map(Y, Z)\big),$$
plays a prominent role in the development of ideas.  This correspondence  identifies a map $F\colon X \times Y \to Z$ with its adjoint $\widehat{F} \colon X \to \map(Y, Z)$ defined by
$$F(x, y) = \widehat{F}(x)(y).$$
If $X$, $Y$, and $Z$ are digital images, we already have a notion of continuity for maps  $F\colon X \times Y \to Z$.  We now define a notion of adjacency in $\map(Y, Z)$, and hence a notion of continuity for maps $X \to \map(Y, Z)$, in such a way that the exponential correspondence preserves continuity.

\begin{definition}\label{def: adjacent functions}
Suppose $X$, $Y$ and $Z$ are digital images.  We define the \emph{digital function space} $\map(Y, Z)$ as the set of all  maps $Y \to Z$ with adjacency as follows:
For $f, g \in \map(Y, Z)$, we say that $f$ and $g$ are \emph{adjacent in $\map(Y, Z)$}, and write $f \sim_{\map(Y, Z)} g$,  if $f(y)\sim_Z g(y')$  whenever $y \sim_Y y'$.

For reasons  that  will emerge below (cf.~\lemref{lem: N-stage homotopy}), we sometimes use the more compact notation $f \approx_1 g$ in place of $f \sim_{\map(Y, Z)} g$, especially if the function space in which $f$ and $g$ are adjacent is clear from context.   Moreover, we say that a \emph{function $G \colon X \to \map(Y, Z)$ is continuous} if $G(x) \approx_1G(x')$  whenever $x \sim_X x'$.
\end{definition}

 \begin{remark}
 In considering digital function spaces, it seems we are passing out of the category of digital images and maps.   However, this fact does not seem to cause problems in our development.  The situation is perhaps comparable to that of ordinary homotopy theory, whereby a function space $\map(Y, Z)$ is generally not a CW complex, and certainly is not a finite-dimensional space, even though $Y$ and $Z$ may be.  Nonetheless, function spaces still play a useful role there.

The \emph{a priori} more general setting of \emph{tolerance spaces} that we mentioned in the introduction does extend to include function spaces as we have defined  them here.
 \end{remark}

 \begin{proposition}[Exponential law]\label{prop: exponential law}
  Suppose $X$, $Y$, and $Z$ are digital images, and that $F\colon X \times Y \to Z$ and $\widehat{F} \colon X \to \map(Y, Z)$ are adjoint under the exponential correspondence, so that
$F(x, y) = \widehat{F}(x)(y)$.  Then $F$ is continuous if and only $\widehat{F}$ is continuous.
 \end{proposition}

 \begin{proof}
 Suppose that $F\colon X \times Y \to Z$ is continuous.  For $x \sim_X x'$, we must show that $\widehat{F}(x)$ and $\widehat{F}(x')$ are adjacent in  $\map(Y, Z)$.  For this, take $y \sim_Y y'$.  In the product $X \times Y$, we have  $(x, y) \sim_{X \times Y} (x', y')$ \emph{per} \defref{def: products}.  Since $F$ is continuous, we have
$\widehat{F}(x)(y) = F(x, y) \sim_Z  F(x', y') = \widehat{F}(x')(y')$.  It follows from \defref{def: adjacent functions} that $\widehat{F}$ is continuous.

Conversely, suppose that $\widehat{F}$ is continuous.  For $(x, y) \sim_{X \times Y} (x', y')$, we must show that $F(x, y) \sim_Z F(x', y')$.  From \defref{def: products}, we have that $x \sim_X x'$  and $y \sim_Y y'$.  Then $\widehat{F}(x)$ and $\widehat{F}(x')$ are adjacent in  $\map(Y, Z)$, since  $\widehat{F}$ is continuous.  Therefore, from \defref{def: adjacent functions}, we have  $F(x, y) = \widehat{F}(x)(y) \sim_Z \widehat{F}(x')(y') = F(x', y')$.  It follows that $F$ is continuous.
 \end{proof}

\begin{lemma}\label{lem: induced functions}
Suppose $f\colon X \to Y$ is a map of digital images, and $Z$ is any digital image.  The induced functions of digital function spaces
\begin{enumerate}
\item $f^* \colon \map(Y, Z) \to \map(X, Z)$, defined by $f^*(g) = g\circ f$, and
\item  $f_* \colon \map(Z, X) \to \map(Z, Y)$, defined by $f_*(g) = f\circ g$,
\end{enumerate}
are both ``continuous," in the sense that they preserve adjacency as we have defined it in  \defref{def: adjacent functions}.
\end{lemma}

\begin{proof}
(1) Suppose we have $g \approx_1 h$ and $x \sim_X x'$.  Then $f(x) \sim_Y f(x')$, since $f$ is continuous, and hence $g\circ f(x) \sim_Z h\circ f(x')$.  That is, we have
$g\circ f \approx_1 h\circ f$, so $f^*$ preserves adjacency.

The proof of (2) is similar.
\end{proof}

We use item (1) of the above very frequently in  \secref{sec: cofibrations} and the sequel. Next, we will define a digital path space as a special case of a digital function space.

For $X$ a digital image and any  $N \geq 1$, a \emph{path of length $N$ in $X$} is a continuous map $\alpha\colon I_N \to X$.  Unlike in the ordinary homotopy setting, where any path may be taken with the fixed domain $[0, 1]$, in the digital setting we must allow paths to have different domains.  The situation is perhaps comparable to taking Moore paths in a topological space.

\begin{definition}[Digital Path Space]\label{def:dig path space}
Let $X$ be a digital image.   For each $N \geq 1$, the \emph{digital path space} (of paths of length $N$ in $X$) $P_NX = \map(I_N, X)$ consists of all paths of length $N$ in $X$, together with the adjacency relation of \defref{def: adjacent functions}.
\end{definition}

So two paths  $\alpha, \beta\colon I_N \to X$ of length $N$ are adjacent if $\alpha(k)$ and $\beta(k')$  are adjacent in $X$, whenever $k \sim k' \in I_N$. Also note that, \emph{per} \defref{def: adjacent functions}, a map $X \to P_NY$ is \emph{continuous} if it is continuous in the usual digital sense, namely, if it preserves adjacency.

\begin{remark}
For our purposes in this paper, we are able to treat paths of different lengths as occupying different path spaces.   It is possible, though, that some situations demand treating paths of different lengths together, as part of a ``unified" path space that includes paths of all lengths.   It is possible to do this, if desired, e.g. in the following way. Define a path in $X$ not as we have done, but rather as a map from the natural numbers $\alpha\colon \N \to X$ that preserves adjacency in the obvious way.  This departs from our conventions, because $\N$ is not a finite digital image, but otherwise does not cause any problems.  (Note, though, that $\N$ would be a tolerance space.) Then we regard a path $\alpha$ to be  of length $N$ if $\alpha(k) = \alpha(k+1)$ for all $k \geq N$ (take the smallest such $N$ if it is desired that each path have a unique length).  It is easy to give a suitable adjacency relation on this unified path space, $P_\infty X$, say,  that allows paths of different lengths to be adjacent, or not.  Furthermore, the fixed-length path spaces we consider here may be included in this $P_\infty X$ in an obvious way  so that adjacency of paths is preserved.   In this way, our path spaces $P_NX$ may be viewed as something like ``skeleta" of $P_\infty X$.        Path spaces in the setting of tolerance spaces are described in the thesis of Poston \cite{Po71}, exactly as we have done here (in both fixed-length and this latter $P_\infty X$ sense).
\end{remark}

\begin{definition}[Evaluation Maps]\label{def:dig eval maps}
Let $Y$ be a digital image.  For each  digital path space  $P_NY$  and $t=0$ or $t = N$,  we have an evaluation map $\mathrm{ev}_t\colon P_NY \to Y$, defined by $\mathrm{ev}_t(\alpha) = \alpha(t)$, for $\alpha \in P_NY$.  We also have  the evaluation map $\pi\colon P_N Y \to Y \times Y$ given by $\pi(\gamma) = \big(\gamma(0), \gamma(N)\big)$, for each $\gamma \in P_N Y$.
\end{definition}

\begin{lemma}
 These evaluation maps are continuous, in the sense that we have $\alpha \sim_{P_NY} \beta \implies \mathrm{ev}_t(\alpha) \sim_Y \mathrm{ev}_t(\beta)$, and $\alpha \sim_{P_NY} \beta \implies \pi(\alpha) \sim_{Y\times Y} \pi(\beta)$.
\end{lemma}

\begin{proof}
Continuity of $\mathrm{ev}_0$ and $\mathrm{ev}_N$ follows directly from the definitions.  Then, we may write $\pi$ as $\pi(\alpha) = (\mathrm{ev}_0(\alpha), \mathrm{ev}_N(\alpha))$, and continuity of $\pi$ follows from  that of $\mathrm{ev}_0$ and $\mathrm{ev}_N$.
\end{proof}

\begin{examples}\label{ex: induced maps}
(1)
We may ``prolong" paths, in the following way.  For any $N \geq M$, we have a map $q\colon I_N \to I_M$ given by
$$q(t) = \begin{cases} t & 0 \leq t \leq M \\ M & M \leq t \leq N. \end{cases}$$
Then we obtain an induced function of path spaces $q^*\colon P_M Y \to P_N Y$ that preserves adjacency, as in \lemref{lem: induced functions}.  For $\alpha \in P_MY$ a path of length $M$, its prolonged version $q^*(\alpha)$ is sometimes referred to as a \emph{trivial extension} of $\alpha$ in the literature (cf.~\cite[Def.4.6]{Bo99}, for instance).

(2) With $\{0\} \subseteq \Z$ a single point, we have an evident identification of $\map(\{0\}, Y)$ and $Y$, for any digital image $Y$ (we have adjacency-preserving bijections in each direction, inverse to each other).   Furthermore, for the inclusion $i \colon \{0\} \to I_N$, any $N\geq 1$ and any $Y$, the induced function
$i^* \colon \map(I_m, Y) \to \map(\{0\}, Y)$ may be identified with the evaluation map $\text{ev}_0 \colon P_NY \to Y$.

(3)  A subdivision of an interval is a longer interval, thus: $S(I_N, k) = I_{kN + k-1}$.  Then the projection $\rho_k \colon S(I_N, k) \to I_N$ induces a function
$$(\rho_k)^*\colon \map(I_N, Y) \to \map(S(I_N, k) , Y)$$
that may equally well be regarded as a function of path spaces $(\rho_k)^*\colon P_N Y \to  P_{kN + k-1} Y$.  Notice that, whilst this also takes paths in $Y$ to longer paths in $Y$, it does so in a way quite different from the trivial extensions of (1).
\end{examples}

The other evaluation maps of \defref{def:dig eval maps} may also be identified with induced functions of mapping spaces as in (2) above.  We will make  use of such identifications, as well as the other observations above, in \secref{sec: fibrations} and developments that follow it.

Now we discuss the notion of homotopy.  As function spaces, our notion of adjacency in a path space here is chosen so as to provide an exponential correspondence.  In ordinary homotopy theory, this correspondence means that a homotopy $H\colon X \times I \to Y$ may be viewed equally well as a map $\widehat{H}\colon X \to \map(I, Y)$ into the path space.
We will give the corresponding  two definitions of homotopy in the digital setting, and then show they are equivalent.

\begin{definition}[Left and Right Homotopy]\label{def: Left-Right homotopy}
Let $X$ and $Y$ be digital images, and $f, g\colon X \to Y$ (continuous) digital maps.

(A) (Cylinder object definition.)
We say that $f$ and $g$ are \emph{left homotopic} if, for some $N\geq 1$, there is a continuous map
$$H \colon X \times I_N \to Y,$$
with $H(x, 0) = f(x)$ and $H(x, N) = g(x)$.  Then $H$ is a left homotopy from $f$ to $g$.

(B) (Path object definition.)  We say that $f$ and $g$ are \emph{right homotopic} if, for some $N\geq 1$, there is a continuous map (in the sense of \defref{def:dig path space})
$$\widehat{H} \colon X  \to P_NY,$$
for which $f = \mathrm{ev}_0\circ \widehat{H}$ and $g = \mathrm{ev}_N\circ \widehat{H}$.  Then $\widehat{H}$ is a right homotopy from $f$ to $g$.
\end{definition}

\begin{remark}
Let $\pi\colon P_N Y \to Y \times Y$ be the evaluation map from \defref{def:dig eval maps}.   It is easy to see that a right homotopy from $f$ to $g$ is equivalent to a  filler $\widehat{H} \colon X  \to P_NY$ in the following commutative diagram:
$$\xymatrix{  &  P_{N} Y \ar[d]^{\pi} \\
X \ar@{.>}[ru]^-{\widehat{H} } \ar[r]_-{(f, g)} & Y\times Y.}$$
 Here, we have written $(f, g)$ for the map $(f \times g)\circ \Delta\colon X \to Y \times Y$, with $\Delta\colon X \to X \times X$ the diagonal map given by $\Delta(x) = (x, x)$.
\end{remark}

By taking adjoints,  we may pass between maps from a cylinder object and maps to a path object.  This provides a correspondence between left and right homotopies.

\begin{proposition}\label{prop: left=right homotopy}
Suppose $f, g \colon X \to Y$ are digital maps.  Then $f$ and $g$ are left homotopic if and only if they are right homotopic.
\end{proposition}

\begin{proof}
Suppose that  $H \colon X \times I_N \to Y$ is a left homotopy from $f$ to $g$. We form the adjoint $\widehat{H} \colon X \to P_NY$ as $\widehat{H}(x)(k) = H(x, k)$.
Then $\widehat{H}$, with the evaluation maps $\mathrm{ev}_0, \mathrm{ev}_N\colon P_NY \to Y$, is a right homotopy from $f$ to $g$.    For we have $\mathrm{ev}_0\circ \widehat{H}(x) = H(x, 0) = f(x)$, and  $\mathrm{ev}_N\circ \widehat{H}(x) = H(x, N) = g(x)$.  The continuity of  $\widehat{H}$ follows from \propref{prop: exponential law}.

Conversely, suppose that  $\widehat{G} \colon X \to P_NY$, together with the evaluation maps $\mathrm{ev}_0, \mathrm{ev}_N\colon P_NY \to Y$, is a right homotopy from $f$ to $g$.    Then the adjoint $G$ of $\widehat{G}$ is defined as $G \colon X \times I_N \to Y$, with
$G(x, k) = \widehat{G}(x)(k)$ for $0 \leq k \leq N$.  It follows from the definitions that $G$ is a left homotopy from $f$ to $g$.
\end{proof}

\begin{definition}[Digital Homotopy]\label{def: digital homotopy (bis)}
We say that digital maps $f, g\colon X \to Y$ are homotopic if, for some $N$, there is a left homotopy $H \colon X \times I_N \to Y$, equivalently a right homotopy $\widehat{H} \colon X \to P_NY$ with the evaluation maps $\mathrm{ev}_0, \mathrm{ev}_N\colon P_NY \to Y$, from $f$ to $g$.  Notice that, from the proof of  \propref{prop: left=right homotopy}, we may use the same $N$ for left or right homotopy.  We write $f \approx g\colon X \to Y$, and think of such a homotopy as an $N$-stage deformation of $f$ into $g$.  Generally, even for given digital images  $X$ and $Y$, $N$ will depend on $f$ and $g$.
\end{definition}

\begin{remark}\label{rem: graph product homotopy}
Homotopy of digital maps has been studied by Boxer and others (see, e.g. \cite{Bo99, Bo05}).  Our definition of left homotopy above is visually the same as that of these authors.  There is a technical difference, however, in that they take the ``graph product" adjacency relation in the product $X \times I_N$, and not the adjacency relation we use (cf.~remarks after Definition 2.5 of \cite{Bo06}).  The difference is akin to requiring a function of two variables to be separately or jointly continuous.  Therefore, our homotopies must preserve more adjacencies than those of  \cite{Bo99}, and this fact has important consequences---see \propref{prop: diamond non-contractible} and the remarks above it.   Note that  various choices of adjacency relation on a product are discussed  in  \cite{Bo18}.
\end{remark}

We may extend the definition of a path in a digital image to that of a path in a function space in an obvious way.  Namely, we say that  a continuous map $\alpha\colon I_N \to \map(X, Y)$---in the sense we have defined such in \defref{def: adjacent functions}---is a \emph{path of length $N$ in $\map(X, Y)$}.  Then, by forming the adjoint of a left homotopy $H\colon X \times I_N \to Y$ as
$$\alpha_{H} \colon I_N \to \map(X, Y), $$
with $\alpha_{H}(t)(x) = H(x, t)$, we see that a homotopy may be viewed as a path in the function space.  The following explains the notation for adjacent functions that we started using above.

\begin{lemma}\label{lem: N-stage homotopy}
Maps $f, g\colon X \to Y$ are homotopic via an $N$-stage homotopy if, and only if, there is a (continuous) path of length $N$ in $\map(X, Y)$ from $f$ to $g$.  In particular, $f, g\colon X \to Y$ are homotopic via a $1$-stage homotopy if, and only if, $f$ and $g$ are adjacent in $\map(X, Y)$.  In this latter case, we write $f \approx_1 g\colon X \to Y$.
\end{lemma}

\begin{proof}
From above, we see that an $N$-stage homotopy $H\colon X \times I_N \to Y$ corresponds to a path $\alpha_{H} \colon I_N \to \map(X, Y)$ of length $N$.  If the homotopy starts at $f$ and ends at $g$, then so does the path, and vice versa. This correspondence preserves continuity, by the exponential law  \propref{prop: exponential law}. It is worth noting the special case $N = 1$.  If $f, g\colon X \to Y$ are adjacent, as we have defined adjacent functions in \defref{def: adjacent functions}, then defining $H\colon X \times I_1 \to Y$ as $H(x, 0) = f(x)$ and $H(x, 1) = g(x)$ gives a $1$-stage homotopy from $f$ to $g$.
\end{proof}

Hence our notation $f \approx_1 g$ in this case.  In principle, we could adopt the notation $f \approx_N g$ to indicate that there is an $N$-stage homotopy from $f$ to $g$, but we have no need of this notation at this time.

\begin{lemma}\label{lem: homotopy equiv reln}
Homotopy of maps is an equivalence relation on the set of all maps $X \to Y$.
\end{lemma}

\begin{proof}
The usual argument (such as that of Proposition 2.8 in  \cite{Bo99}) suffices.  We just have to be careful that the technical point mentioned in \remref{rem: graph product homotopy} above does not cause problems.  Reflexivity and symmetricity are immediate.  For transitivity,  say we have a homotopy $H \colon X \times I_N \to Y$ from $f$ to $g$, and a homotopy $G\colon X \times I_M \to Y$ from $g$ to $h$.   We assemble a putative homotopy $F\colon X \times I_{N+M} \to Y$ from $f$ to $h$, defined by
$$F(x, t) = \begin{cases}  H(x, t) & 0 \leq t \leq N \\ G(x, t-N) & N \leq t \leq N+M. \end{cases}$$
We must check that $(x, t) \sim_{X \times I_{N+M}} (x' t')$ implies   $F(x, t) \sim_{Y} F(x' t')$.  But we must have $t \sim_{I_{N+M}} t'$, so  they differ by at most $1$.  Thus either we have both $t$ and $t'$ in $I_N$, or both $t$ and $t'$ in $\{N, \ldots, N+M\}$. In the first case, continuity of $H$ gives   $F(x, t) \sim_{Y} F(x' t')$.  In the second, continuity of $G$ gives  the same.
\end{proof}

\begin{definition} Let $f \colon X \to Y$ be a map of digital  images.  If there is a map $g \colon Y \to X$ such that $g\circ f \approx \text{id}_X$ and $f \circ g \approx \text{id}_Y$, then $f$ is a \emph{homotopy equivalence}, and $X$ and $Y$ are said to be \emph{homotopy equivalent}, or to have the same \emph{homotopy type.}
\end{definition}

\begin{definition}\label{def:contractible}
A digital image is \emph{contractible (to a point)} if it is homotopy equivalent to a point.  Notice that this is equivalent to saying there is some $x_0 \in X$ and some $N$,  for which we have a homotopy
$H\colon  X \times I_N \to  X$  with $H(x, 0) = x$, and $H(x, N) = x_0$.
\end{definition}

\begin{example}
Any interval $I_M$ is contractible to a point.  Indeed, the homotopy $H\colon I_M \times I_M \to I_M$ defined by
$$H(s, t) = \begin{cases} s & 0 \leq s \leq M - t\\ M-t & M - t \leq s \leq M\end{cases}$$
begins at $H(s, 0) = s$, which is the identity $\text{id}\colon I_M \to I_M$, and ends at $H(s, M) = 0$, which is the constant map at $0 \in I_M$.

More generally, if $X$ and $Y$ are contractible digital images, then their product $X \times Y$ is also contractible.  For suppose $H\colon X \times I_M \to X$ and $G\colon Y \times I_N \to Y$ are contracting homotopies, so that $H(x, 0) = x$ and $H(x, M) = x_0$, and $G(y, 0) = y$ and $G(y, N) = y_0$.  Without loss of generality, we may assume that we have $M = N$.  For if $M \not= N$, we may prolong the shorter homotopy by a constant homotopy. For instance, if we have $M < N$, then define $H'\colon X \times I_N \to X$ by
$$H'(x, t) = \begin{cases} H(x, t) & 0 \leq t \leq M \\
x_0 & M \leq t  \leq N.\end{cases}$$
This is a special case of the situation we considered when establishing transitivity in \lemref{lem: homotopy equiv reln}; this prolonged homotopy is continuous by that argument.  So assume we have $M = N$, and define $\mathcal{H}\colon X \times Y \times I_M \to X \times Y$ by $\mathcal{H}(x, y, t) = \big( H(x, t), G(y, t)\big)$.  This is easily checked to be a contracting homotopy for $X \times Y$.  So, for instance, any product of intervals, such as an $n$-cube $(I_M)^n \subseteq \Z^n$ is contractible.
\end{example}

Obviously, we are concerned to have plenty of non-contractible digital images, too.  In the continuous setting, the first such example would normally be a circle. In the digital setting, because our notion of homotopy equivalence is such a rigid one, ``circles" of different sizes are generally non-homotopy equivalent to each other.  Indeed, it is not so clear that we are able to give a good definition of a ``circle up to homotopy" that includes the kinds of digital images that one might want to be considered equivalent to a circle (see the related comments in \secref{sec: future}).  Still, it seems reasonable to consider  \emph{the Diamond} $D \subseteq \Z^2$ as a digital circle. Recall that this consist of the four points
$$D = \{ (1, 0), (0, 1), (-1, 0), (0, -1) \},$$
and is pictured in \figref{fig:D & C} of \exref{ex:basic digital images}.

In \remref{rem: graph product homotopy} above we indicated that, when defining homotopy, using the ``graph product" adjacencies in $X \times I_N$, rather than the adjacencies that we use,  has important consequences.  A fundamental difference between the two conventions appears here.  In \cite{Bo99} (following Th.3.1 there), it is shown that, using the notion of homotopy that derives from the ``graph product,"  the Diamond is contractible.  However, using the notion of homotopy as we have defined it, the contracting homotopy used in \cite{Bo99} fails to be continuous.  In fact, by contrast,  we have the following.

\begin{proposition}\label{prop: diamond non-contractible}
The Diamond $D$ is not contractible.
\end{proposition}

\begin{proof}
For suppose that we have $H\colon D \times I_N \to D$ that satisfies $H(x, 0) = x$ for $x \in D$, and $\{ H(x, N) \mid x \in D\}$ omits at least one point from $D$.   We assume that $H$ is continuous, and arrive at a contradiction.   There must be some first time $t$ at which we have $\{ H(x, t) \mid x \in D\} = D$, and $\{ H(x, t+1) \mid x \in D\}$ omits at least one point from $D$. Without loss of generality, suppose that   $\{ H(x, t+1) \mid x \in D\}$ does not include $(1, 0)$---the other choices are handled with an identical argument.   Suppose that, at time $t$, we have
$$H(x_1, t)  = (1, 0), \quad H(x_2, t)  = (0, 1), \quad H(x_3, t)  = (-1, 0), \quad H(x_4, t)  = (0, -1), $$
with $\{ x_1, x_2, x_3, x_4 \} = D$.  Since $(0, 1) \not\sim_D (0, -1)$, we must have $x_2 \not\sim_D x_4$, and thus $x_2 \sim_D x_1 \sim_D x_4$.  Since $H$ is continuous, we must have $H(x_1, t+1)\sim_D H(x_1, t) = (1, 0)$.  So either $H(x_1, t+1) = (0, 1)$, or $H(x_1, t+1) = (0, -1)$ (recall that $(1, 0)$ is not in the image of $H$ at time $t+1$).  If   $H(x_1, t+1)=(0, 1)$, then $(x_4, t) \sim_{D \times I_N} (x_1, t+1) \implies H(x_4, t) \sim_{D} H(x_1, t+1)$, or $(0, -1) \sim_D (0, 1)$, a contradiction.  If   $H(x_1, t+1)=(0, -1)$, then $(x_2, t) \sim_{D \times I_N} (x_1, t+1) \implies H(x_2, t) \sim_{D} H(x_1, t+1)$, or $(0, 1) \sim_D (0, -1)$, again a contradiction.
\end{proof}

\begin{remark}
We may extend the notion of left homotopy to one of homotopy of maps into a path space in an obvious way.  Namely, a  continuous map
$$H \colon X \times I_M \to P_NY$$
is a homotopy from $H(-,0)\colon X \to P_NY$ to $H(-,N)\colon X \to P_NY$.  With the adjunction used in the proof of  \propref{prop: left=right homotopy}, such a map may be viewed as a homotopy of homotopies.  Furthermore, we may also discuss continuity and homotopy for maps between path spaces: a continuous map in these contexts means an adjacency-preserving function.
\end{remark}

As a positive example of homotopy equivalent spaces in the digital setting, we offer the following.  Notice that this result, and its proof, mirror the corresponding homotopy equivalence in the topological setting.

\begin{proposition}
For any digital image $Y$ and any $N$, the evaluation map $\text{ev}_0 \colon P_NY \to Y$ is a homotopy equivalence.
\end{proposition}

\begin{proof}
Define $\sigma \colon Y \to P_NY$ as $\sigma(y) = c_y$, with $c_y\colon I_N \to Y$ the constant path at $y \in Y$. Clearly $\sigma$ preserves adjacency, and so is continuous in the appropriate sense.  Then $\sigma$ is a right inverse to $\text{ev}_0$: we have
$$\text{ev}_0 \circ \sigma = \text{id}_Y\colon Y \to P_N Y \to Y.$$
Now define a homotopy---in the sense of the above remark---$H\colon P_NY \times I_N \to P_N Y$ as
$$H(\alpha, s)(t) = \begin{cases} \alpha(t) & 0 \leq t \leq N-s\\ \alpha(N-s) & N-s \leq t \leq N,\end{cases}$$
for $\alpha \in P_NY$, $s \in I_N$, and $0 \leq t \leq N$. Obviously we have $H(\alpha, 0) = \alpha$, and  $H(\alpha, N) = c_{\alpha(0)} = \sigma\circ\text{ev}_0(\alpha)$.  So it remains to check that $H$ preserves adjacency.  To this end, suppose we have $\alpha \sim_{P_NY} \alpha'$ and $s \sim_{I_N} s'$.  We must check that $H(\alpha, s) \sim_{P_N Y} H(\alpha', s')$, which entails checking $H(\alpha, s)(t) \sim_{Y} H(\alpha', s')(t')$ whenever we have $t \sim_{I_N} t'$.
Write $I_N \times I_N = \{ (s, t) \mid 0 \leq s, t, \leq N\}$.  Our formula for $H$ means that, for $(s, t) \in I_N$ with $s + t \leq N$, we use $\alpha(t)$ to evaluate $H$, and when $(s, t) \in I_N$ with $N \leq s + t \leq 2N$, we use $\alpha(N-s)$ to evaluate $H$.   Now $s+t$ and $s' + t'$ may differ by no more than two, if we have $s \sim_{I_N} s'$ and $t \sim_{I_N} t'$.  Hence, we have three possibilities: (1) we have both $s+t$ and $s'+t'$ in $\{0, \ldots, N\}$; (2) both $s+t$ and $s'+t'$ in $\{N, \ldots, 2N\}$; or (3) $\{s+t, s'+t'\} = \{N-1, N+1\}$. In the first case, we have   $H(\alpha, s)(t) = \alpha(t) \sim_Y \alpha'(t') =  H(\alpha', s')(t')$.  In the second case, we have $H(\alpha, s)(t) = \alpha(N-s) \sim_Y \alpha'(N - s') =  H(\alpha', s')(t')$ (notice here that, since $s \sim_{I_N} s'$, we have $N - s \sim_{I_N} N - s'$).   It remains to check the third case,  in which $\{s+t, s'+t'\} = \{N-1, N+1\}$ and so both formulas are used to evaluate $H$.  Suppose that we have $s+t = N-1$ and $s'+t' = N+1$, which entails that we have $t' = t+1$ and $s' = s+1$.  Then we have $N - s' = t' -1 = t$, and so
$$H(\alpha, s)(t) = \alpha(t) \sim_Y \alpha'(t) = \alpha'(N - s') = H(\alpha', s')(t').$$
On the other hand, if we have  $s+t = N+1$ and $s'+t' = N-1$, so $t' = t-1$ and $s' = s-1$ and so $N - s = t -1 = t'$, then
$$H(\alpha, s)(t) = \alpha(N-s) \sim_Y \alpha'(N-s) = \alpha'(t') = H(\alpha', s')(t').$$
Thus $H$ preserves adjacency, as required.
\end{proof}

Generally, though, this notion of homotopy equivalence is a very rigid one and many examples of homotopy equivalent spaces from the continuous setting fail to transfer as such into the digital setting---see comments in the final section.

\section{Digital Cofibrations}\label{sec: cofibrations}

None of the material in this section and the next has appeared in the digital topology literature before.

Recall that, in the topological setting, a \emph{cofibration} $A \to X$ is a map that has the homotopy extension property.    This property may be expressed diagramatically as follows.   For any $Y$, let $PY$ denote the space of (unbased) paths in $Y$, and denote by $\mathrm{ev}_0\colon PY \to Y$ the map that evaluates a path at its initial point; thus we have $\mathrm{ev}_0(\gamma) = \gamma(0)$, for $\gamma$ a path in $Y$.  Then  $j\colon A \to X$ is a cofibration when, for any $Y$ the following  commutative diagram has a filler $\bar{H}\colon X \to PY$.  Namely, the homotopy $H\colon A \to PY$ extends to a homotopy $\bar{H}\colon X \to PY$ that begins at the map $f \colon X \to Y$:
$$\xymatrix{
A \ar[r]^-{H} \ar[d]_-{j}  & PY \ar[d]^-{\mathrm{ev}_0}\\
X \ar[r]_-{f}  \ar@{.>}[ru]^{\bar{H}} & Y   }
$$

Unfortunately, if we try to repeat this definition in the digital setting, it leads to many inclusions failing to qualify as a cofibration.  The following simple example illustrates the issue.

\begin{example}\label{ex: non-cofibration}
Corresponding to the above ingredients,  take digital images  $A = \{0\}$, $X = I_2 = \{0, 1, 2\}$, and $Y = I_3 = \{0, 1, 2, 3\}$.  Let $j \colon A \to X$ be the obvious inclusion of $0$ into the digital interval of length $2$.     Define maps
$f\colon X \to Y$ and $H\colon A  \to P_1Y = \map(I_1, Y)$ by $f(k) = k+1$ for $k = 0, 1, 2$, $H(0)(l) = 1-l$ for $l = 0, 1$.  This gives the following commutative diagram.
$$\xymatrix{
\{0\} \ar[r]^-{H} \ar[d]_-{j}  & P_1I_3 \ar[d]^-{\mathrm{ev}_0}\\
I_2 \ar[r]_-{f}  \ar@{.>}[ru]^{\bar{H}} & I_3   }
$$
We claim that there is no (digitally continuous)  filler $\bar{H}\colon I_2  \to P_1I_3$ for the diagram.  This follows because such a filler is equivalent, via adjoints, to a map
$\widehat{H}\colon I_2 \times I_1 \to I_3$ with
$$\widehat{H}(k, 0) = \bar{H}(k)(0) = f(k) = k+1 \qquad \text{and} \qquad \widehat{H}(0, l) = \bar{H}(0)(l) = H(0(l) = 1-l,$$
for $k = 0, 1, 2$ and $l = 0, 1$.
Since
$(1,1) \sim_{I_2 \times I_1} (0, 1)$ and $(1,1) \sim_{I_2 \times I_1} (2, 0)$, such (an adjoint of) a filler would need to satisfy both
$$\widehat{H}(1, 1) \sim_{I_3} \widehat{H}(0, 1) =  0  \qquad \text{and} \qquad \widehat{H}(1, 1) \sim_{I_3} \widehat{H}(2, 0) =  3.$$
But there is no element in $I_3$ adjacent to both $0$ and $3$.  Thus there is no filler.
\end{example}

\begin{remark}
A notion of cofibration (or adjunction space) in the tolerance space setting is given in \cite{Po71}.  However, as we have pointed out, repeating the continuous definition gives a notion that is too rigid to be of much practical use.  Poston gives an example similar to \exref{ex: non-cofibration}, and remarks that developing a notion of cell complexes in the tolerance setting is not likely to be of much use, because of this rigidity.  Whereas Poston sees cofibrations mainly as a way of developing cell complexes, we are interested in them here as a source of fibrations---or, certain maps that have a homotopy lifting property (see \secref{sec: fibrations}).  Furthermore, incorporating subdivision into our notion of cofibration, as we do below, is the point of departure from previous appearances of cofibration in a digital (or tolerance) setting, and it is this that allows us to develop the notion in a way that has substantial application and depth.
\end{remark}

Motivated by the desire to have (at least) the inclusion $\{0\} \to I_M$ be a ``digital cofibration," we define this notion in a way that relaxes, or makes less rigid, the idea of extending a homotopy.  The way we do this involves the notion of subdivision, from \secref{sec: basics}.   In the following, $\mathrm{ev}_0\colon P_NY \to Y$ denotes the evaluation map from \defref{def:dig eval maps} that evaluates an unbased  path at its initial point.

\begin{definition} [Digital cofibration]\label{def: digital cofibration}
An inclusion of digital images $j \colon A \to X \subseteq \Z^n$ is a \emph{cofibration} if, given a commutative diagram
$$\xymatrix{
A \ar[r]^-{H} \ar[d]_-{j}  & P_NY \ar[d]^-{\mathrm{ev}_0}\\
X \ar[r]_-{f}   & Y   }
$$
(any $N$ and any digital image $Y$),
there are subdivisions $S( X, k)$ and $S(I_N, l)$, and  a filler $\overline{H}\colon S( X, k) \to \map( S(I_N, l), Y)$ in the following commutative diagram:
$$\xymatrix{
S(A, k) \ar[d]_{S(j,k)} \ar[r]^-{\rho_k} & A \ar[r]^-{H} \ar[d]_-{j}  & P_NY \ar[d]^-{\mathrm{ev}_0} \ar[r]^-{ (\rho_l)^* } &  \map( S(I_N, l), Y) = P_{lN + l-1} Y \ar[ld]^{\mathrm{ev}_0}\\
S(X, k) \ar[r]_-{\rho_k}  \ar@{.>}[rrru]^-{\overline{H}} &X \ar[r]_-{f}   & Y   }
$$
\end{definition}

Note that, in the above diagram, the function
$$(\rho_l)^*\colon P_NY = \map(I_N, Y) \to \map( S(I_N, l), Y) = \map(I_{lN + l-1}, Y) = P_{lN + l-1} Y$$
is that induced by pre-composition with the projection $\rho_l\colon S(I_N, l) \to I_N$, as in \lemref{lem: induced functions}.
The reason for the form of this definition should become clear over the course of the next several results.

\begin{discussion}\label{dis: cofibn pushout}
In the topological setting, suppose that we have $j\colon A \to X$ the inclusion of a closed subspace.  Then the commutative diagram
$$\xymatrix{A \ar[d]_{j} \ar[r]^-{i_1} & A \times I \ar[d]^{j\times \mathrm{id}} \\
X \ar[r]_-{i_1} & A\times I \cup X \times \{0\}}$$
is a pushout, that is, given maps $H \colon A \times I \to Y$ and $f \colon X \to Y$ that agree on $A$, so that we have $H(-, 0) = f\circ j\colon A \to Y$, then there is a (unique) filler $\phi$  in  the commutative diagram
$$\xymatrix{A \ar[d]_{j} \ar[r]^-{i_1} & A \times I \ar[d]_{j\times \mathrm{id}}  \ar@/^1pc/[rdd]^{H} \\
X \ar[r]^-{i_1} \ar@/_1pc/[rrd]_{f}& A\times I \cup X \times \{0\} \ar@{..>}[rd]_{\phi}\\
 & & Y}$$
Here, the issue is a \emph{continuous} filler: there is only one candidate, namely $\phi|_{X \times \{0\}} = f$ and $\phi|_{A \times I} = H$.  Because we assume $A$ closed in $X$, these maps piece together well.    Taking $H = j \times \text{id}\colon A \times I \to X \times I$ and $f = i_1\colon X \to X \times I$ in the pushout diagram, the filler is  the inclusion
$i\colon A\times I \cup X \times \{0\} \to X \times I$.  For $j\colon A \to X$ the inclusion of a closed subspace, we have $j$ is a cofibration iff this inclusion $i$ admits a left inverse.  That is, $j$ is a cofibration iff $A\times I \cup X \times \{0\}$ is a retract of $X \times I$.  Constructing retracts of this form provides many basic examples of cofibrations.
\end{discussion}

The situation described in the above discussion does not carry over verbatim to the digital setting (see \exref{ex: non-push} below).  Rather, we have the following adaptation to the digital setting.

\begin{lemma}[Digital Pushout]\label{lem:pushout}
Let $j \colon A \to X \subseteq \Z^n$ be an inclusion of digital images.   Suppose we have  a commutative diagram as in below left (namely,  maps $H \colon A \times I_N \to Y$ and $f \colon X \to Y$ that agree on $A$).
$$\xymatrix{A \ar[d]_{j} \ar[r]^-{i_1} & A \times I_N \ar[d]^{H} \\
X \ar[r]_-{f} & Y}
\qquad
\xymatrix{A \ar[d]_{j} \ar[r]^-{i_1} & A \times S(I_N, l) \ar[d]_{\mathrm{incl.}}  \ar@/^1.5pc/[rdd]^{H\circ(\mathrm{id} \times \rho_l)} \\
X \ar[r]^-{i_1} \ar@/_1pc/[rrd]_{f}& A\times S(I_N, l) \cup X \times \{0\} \ar@{..>}[rd]_{\phi}\\
 & & Y}$$
Then, for any subdivision  $S(I_N, l)$ with  $l \geq 2$,  there is a (unique) filler $\phi$ in  the commutative diagram as in above right.
\end{lemma}

\begin{proof}
As in the discussion above, the issue here is continuity: the only candidate for a filler is  $\phi|_{X \times \{0\}} = f$ and $\phi|_{A \times S(I_N, l)} = H\circ (\text{id} \times \rho_l)$. Now, as illustrated in the example below, the only possible problem with continuity arises when we have points in $A\times S(I_N, l) - i_1(A)$ adjacent to points in $X\times \{0\} - i_1(A)$.  So consider a point $(a, q) \in A\times S(I_N, l) - i_1(A)$, so that $a \in A$ and $q \geq 1$, and a point
$(x, 0) \in X\times \{0\} - i_1(A)$.  If these are adjacent---recall that we are in $\Z^{n+1}$,  then we have $q = 1$, and $a \sim_{X} x$.  Then  $\phi(a, q) = \phi(a, 1) = H\circ(\text{id} \times \rho_l)(a, 1) = H(a, 0) = f(a)$, and $\phi(x, 0) = f(x)$.  Now $f(a)  \sim_Y f(x)$, since $a \sim_{X} x$, and so we have $\phi(a, q)  \sim_Y \phi(x,0)$: the filler is continuous.
\end{proof}

\begin{example}\label{ex: non-push}
Suppose that $D = \{ (0, \pm1), (\pm1, 0) \} \subseteq \Z^2$ is the Diamond from \exref{ex:basic digital images} and \propref{prop: diamond non-contractible}.  Let $\alpha, \beta\colon I_1 \to D$ be the digital paths given by $\alpha(0) = (1,0)$ and $\alpha(1) = (0,1)$, and $\beta(0) = (1,0)$ and $\beta(1) = (0,-1)$.  If we take $j\colon A \to X$ to be $j\colon \{0\} \to  I_1$, $H\colon A \times I_N \to Y$ to be $H\colon \{0\} \times I_1 \to D$ defined by $H(0, k) = \alpha(k)$, for $k = 0, 1$,  and $f\colon X \to Y$ to be $\beta\colon I_1 \to D$, then we have a commutative diagram as above left:
$$\xymatrix{\{0\} \ar[d]_{j} \ar[r]^-{i_1} & \{0\} \times I_1 \ar[d]^{H} \\
I_1 \ar[r]_-{f} & D.}$$
First consider a flller for the following diagram (such would exist in the continuous situation):
$$\xymatrix{\{0\} \ar[d]_{j} \ar[r]^-{i_1} & \{0\} \times I_1 \ar[d]_{j\times \mathrm{id}_{I_1}}  \ar@/^1pc/[rdd]^{H} \\
I_1 \ar[r]^-{i_1} \ar@/_1pc/[rrd]_{f}& \{0\}\times I_1 \cup I_1 \times \{0\} \ar@{..>}[rd]_{\phi}\\
 & & D}$$
The only candidate for $\phi$ must satisfy $\phi(1, 0) \sim_Y \phi(0, 1)$, since $(1, 0) \sim_{\Z^2} (0, 1)$ and $\{0\}\times I_1 \cup I_1 \times \{0\}  \subseteq \Z^2$.  But we have $\phi(1, 0) = f(1) = \beta(1) = (0, -1)$, and $\phi(0, 1) = H(0, 1) = \alpha(1) = (0, 1)$.  Since $(0, -1) \not\sim_Y (0, 1)$,  there is no filler.

On the other hand, for any $l \geq 2$, we have a filler for the diagram
$$
\xymatrix{\{0\} \ar[d]_{j} \ar[r]^-{i_1} & \{0\} \times S(I_1, l) \ar[d]_{\mathrm{incl.}}  \ar@/^1.5pc/[rdd]^{H\circ(\mathrm{id} \times \rho_l)} \\
I_1 \ar[r]^-{i_1} \ar@/_1pc/[rrd]_{f}& \{0\}\times S(I_1, l) \cup I_1 \times \{0\} \ar@{..>}[rd]_{\phi}\\
 & & D.}$$
We define $\phi$ as in \lemref{lem:pushout}, by setting $\phi(0, t) =  H(0, \rho_l(t))$, and $\phi(s, 0) =   f( s)$,  for $s \in  \{0, 1\} = I_1$ and $t \in  \{0, 1, 2, 3, \ldots, 2l-1\} = S(I_1, l)$. Now the only possible source of discontinuity in piecing together $\phi$ from $H$ and $f$, here, is that we require $\phi(1, 0) \sim_Y \phi(0, 1)$.   But we have $\phi(0, 1) = H(0, \rho_k(1)) = H(0, 0) = \alpha(0) = \beta(0)$, and
$\phi(1,0) = f(1) = \beta(1) \sim_Y \beta(0)$.
So $\phi$ is the desired filler.
\end{example}

\begin{remark}
\exref{ex: non-push} indicates a difference between the digital and the tolerance settings.  In the tolerance setting, pushouts are straightforward (both pointed and unpointed---see, e.g. \cite{Po71}).  Here, however, the fact that our digital images are always in some ambient $\Z^n$ seems to play a role in constraining, e.g., the notion of pushout.
\end{remark}

The next result  establishes a digital version of the characterization of (topological) cofibrations indicated in \disref{dis: cofibn pushout}.  For an inclusion of digital images $j \colon A \to X$, the diagram
$$\xymatrix{A \ar[d]_{j} \ar[r]^-{i_1} & A \times I_N \ar[d]^{j \times \text{id}} \\
X \ar[r]_-{i_1} & X \times I_N} $$
leads to a map $\phi_i \colon X \times \{0\} \cup A \times S(I_N, l) \to X \times I_N$, as in \lemref{lem:pushout}.  (In this case, actually, we also have an inclusion $X \times \{0\} \cup A \times I_N \to X \times I_N$, as usual, but our general framework demands that we consider $\phi_i$.)

\begin{proposition}\label{prop: retraction iff cof}
Let $j \colon A \to X \subseteq \Z^n$ be an inclusion of digital images.  The following are equivalent:
\begin{enumerate}

\item $j$ is a cofibration;

\item for each $I_N$, there are subdivisions $S(X, k)$ and $S(I_N, l)$ with $l \geq 2$, and a ``retraction" of  the above $\phi_i$
$$R\colon S(X, k) \times S(I_N, lm) \to X \times \{0\} \cup A \times S(I_N, l),$$
in the sense that the diagram
$$\xymatrix{ S(X, k) \times \{0\}  \cup S(A, k) \times S(I_N,  lm)  \ar[r]^-{\mathrm{incl.}} \ar[rd]_-{\rho_k \times \rho_m} & S(X, k) \times S(I_N, lm) \ar[d]^{R}\\
 & X \times \{0\} \cup A \times S(I_N, l) }$$
 commutes, for some further subdivision $S(I_N,  lm) \cong S(S(I_N, l), m)$ of $S(I_N, l)$.  In the diagram, $\mathrm{incl.}$ denotes the obvious inclusion map that restricts to
 $S(j, k) \times \text{id}$ on  $S(A, k) \times S(I_N,  lm)$ and to $\text{id} \times i$ on  $S(X, k) \times \{0\}$.
\end{enumerate}
\end{proposition}

\begin{proof} (1) $\implies$ (2):
Suppose that $j \colon A \to X$ is a cofibration.  Write the adjoint of the inclusion $A \times S(I_N, l) \to X \times \{0\} \cup A \times S(I_N, l)$ as
$$H \colon A \to P_M\big(X \times \{0\} \cup A \times S(I_N, l)\big),$$
so that we have $H(a)(t') = (a, t')$ for typical points $a \in A$ and $t' \in S(I_N, l) = I_M$, with $M  = Nl+l-1$.  Then we have a commutative diagram
$$\xymatrix{
A \ar[r]^-{H} \ar[d]_-{j}  & P_M\big(X \times \{0\} \cup A \times S(I_N, l)\big) \ar[d]^-{\mathrm{ev}_0}\\
X \ar[r]_-{i_1}   & X \times \{0\} \cup A \times S(I_N, l),   }
$$
hence subdivisions $S( X, k)$ and $S(I_M, m)=I_{mM+m-1}$, and  a filler $\overline{H}$ in the following commutative diagram:
$$\xymatrix{
S(A, k) \ar[d]_{S(j,k)} \ar[r]^-{\rho_k} & A \ar[r]^-{H} \ar[d]_-{j}  & P_M\big(X \times \{0\} \cup A \times S(I_N, l)\big)  \ar[d]^-{\mathrm{ev}_0} \ar[r]^-{ (\rho_m)^* } &  P_{mM+m-1}(X \times \{0\} \cup A \times S(I_N, l)) \ar[ld]^{\mathrm{ev}_0}\\
S(X, k) \ar[r]_-{\rho_k}  \ar@{.>}[rrru]^(0.3){\overline{H}} &X \ar[r]_-{i_1}   & X \times \{0\} \cup A \times S(I_N, l)  }
$$
Note that, in the upper right entry, we have written $ \map( S(I_M, m), X \times \{0\} \cup A \times S(I_N, l)) =  \map( I_{mM+m-1}, X \times \{0\} \cup A \times S(I_N, l))$ as
$P_{mM+m-1}(X \times \{0\} \cup A \times S(I_N, l))$.
Write the adjoint of $\overline{H}$ as
$$R\colon S(X, k) \times S(I_M, m) \to X \times \{0\} \cup A \times S(I_N, l),$$
so that we have $R(x', t'') = \overline{H}(x')(t'')$, for typical points $x' \in S(X, k)$ and $t'' \in S(I_M, m)$.  We check that $R$ is a ``retraction" in  the sense given in the enunciation.  For $( a', t'') \in S(A, k) \times S(S(I_N, l), m) = S(A, k) \times S(I_N, lm)$, we have
$$
\begin{aligned}
R\circ \text{incl.}( a', t'') & =  \overline{H}\big(S(j, k)(a')\big)(t'') =  H\big(\rho_k(a')\big)\big(\rho_m(t'')\big) \\
&= \big(\rho_k(a'), \rho_m(t'')\big) = (\rho_k \times \rho_m)(a', t'').
\end{aligned}
$$
For $(x', 0) \in S(X, k) \times \{0\}$, we have
$$
R\circ \text{incl.}( x', 0) = \overline{H}(x')(0) = i_1\circ \rho_k(x') = ( \rho_k(x'), 0) = (\rho_k \times \rho_m)(x', 0).$$

(2) $\implies$ (1):  Assume that, for each $I_N$,  we have the subdivisions and a retraction---in the sense of the enunciation, and suppose we are given a commutative diagram
$$\xymatrix{
A \ar[r]^-{H} \ar[d]_-{j}  & P_NY \ar[d]^-{\mathrm{ev}_0}\\
X \ar[r]_-{f}   & Y.   }
$$
The adjoint of $H$ gives a map $\widehat{H}\colon A \times I_N \to Y$, by $\widehat{H}(a, t) = H(a)(t)$. Also, setting $f(x, 0) = f(x)$, gives a map  $f\colon X \times \{0\} \to Y$ that agrees with $\widehat{H}$ on the intersection $A \times I_N \cap X \times \{0\} = A \times \{0\}$.  So by \lemref{lem:pushout},  we have a well-defined, continuous map $\phi \colon A\times S(I_N, l) \cup X \times \{0\} \to Y$, for any $l \geq 2$. Precomposing this map with the given $R$ provides (the adjoint of) the desired filler in  \defref{def: digital cofibration}. So we define
$$\overline{H} \colon S(X, k) \to \map\big(S(I_N, lm),  Y\big)$$
by $\overline{H}(x')(t'') = \phi\circ R(x', t'')$, for typical points $x' \in S(X, k)$ and $t'' \in S(I_N, lm) = S\big( S(I_N, l), m)$.  Finally, we check that this $\overline{H}$ provides a filler in in following commutative  diagram:
$$\xymatrix{
S(A, k) \ar[d]_{S(j,k)} \ar[r]^-{\rho_k} & A \ar[r]^-{H} \ar[d]_-{j}  & P_NY \ar[d]^-{\mathrm{ev}_0} \ar[r]^-{ (\rho_{lm})^* } &  \map( S(I_N, lm), Y) \ar[ld]^{\mathrm{ev}_0}\\
S(X, k) \ar[r]_-{\rho_k}  \ar@{.>}[rrru]^-{\overline{H}} &X \ar[r]_-{f}   & Y   }
$$
For the upper left triangle, using the definitions and properties of the various maps involved, we have
$$
\begin{aligned}
\overline{H}\circ S(j,k)(a')(t'') &= \phi\circ R( S(j,k)(a'), t'') = \phi\circ R \circ(S(j,k) \times \text{id})(a', t'')\\
&= \phi\circ(\rho_k \times \rho_m)(a', t'') =  \phi\big( \rho_k(a'), \rho_m(t'') \big) \\
&= \widehat{H}\circ(\text{id}\times \rho_l)\big( \rho_k(a'), \rho_m(t'') \big)\\
&= H\big( \rho_k(a') \big) \big( \rho_{lm}(t'') \big) = ( \rho_{lm})^*\circ H \circ \rho_k(a')(t''),
\end{aligned}
$$
so this part of the diagram commutes.  For the lower right triangle, we have
$$
\begin{aligned}
\overline{H}(x')(0) &= \phi\circ R( x', 0) = \phi\circ R \circ(\text{id} \times i)(x', 0)\\
&= \phi\circ(\rho_k \times \rho_m)(x', 0) =  \phi\big( \rho_k(x'), 0\big) = f \circ \rho_k(x').
\end{aligned}
$$
So $\overline{H}$ is indeed the desired filler, and $j \colon A \to X$ is a cofibration.
\end{proof}

We are now able to prove the desired result that we discussed leading up to \defref{def: digital cofibration}.

\begin{theorem}\label{thm: well-pointed interval}
For any $M$, the inclusion $j\colon \{0\} \to I_M$ is  a cofibration.
\end{theorem}

\begin{proof}
We proceed using \propref{prop: retraction iff cof}. For this we seek, for each $I_N$,  subdivisions and a retraction (in the sense of \propref{prop: retraction iff cof})
$$R \colon S(I_M, k) \times S(I_N, lm) \to    I_M \times \{0\} \cup \{0\} \times S(I_N, l)$$
with $l \geq 2$.  It is sufficient to use $k = l = m = 2$.  We will do so, and construct a suitable
$$R \colon S(I_M, 2) \times S(I_N, 4) \to    I_M \times \{0\} \cup \{0\} \times S(I_N, 2).$$
Notice that this may be viewed as a map $R \colon I_{2M+1} \times I_{4N+3} \to I_M \times \{0\} \cup \{0\} \times I_{2N+1}$, so visually we want to retract a rectangle onto its (contracted) left and bottom edges.   In the topological setting, a retraction of $I \times I$ onto its left and bottom edges is achieved by mapping points that lie on the diagonal line $y = x + c$ either to $(0,c)$, if $c \geq 0$, or to $(-c, 0)$ if $c \leq 0$.  In  the digital setting, however, this map fails to be continuous for the same reasons on display in   \exref{ex: non-cofibration}.  Furthermore,  the technical requirement that $R$ be a ``retraction" as in  \propref{prop: retraction iff cof} means that we must adapt the approach used in the topological setting  a little.

Specifically, we will use the diagonal retraction from the continuous setting first  to retract  $I_{2M+1} \times I_{4N+3}$ onto $\{0, 1\} \times  I_{4N+3} \cup   I_{2M+1} \times \{0\}$, and then follow this with the standard projections $\rho_2 \times \rho_2$ to arrive at $I_M \times \{0\} \cup \{0\} \times S(I_N, 2)$.  Even though the first step itself is not continuous, the composition of the two steps will, in fact, be continuous.

In terms of formulas, a typical point in $S(I_M, 2) \times S(I_N, 4)$ has coordinates $(p, q)$ with $0 \leq p \leq 2M+1$ and $0 \leq q \leq 4N+3$.
Define a function
$$D \colon S(I_M, 2) \times S(I_N, 4) \to \{0, 1\} \times  S(I_N, 4) \cup   S(I_M, 2) \times \{0\}$$
as
$$D(p, q) =
\begin{cases}
(p, q) & p \leq 1 \\
(1, q - p+1) & p \geq 1 \text{ and }  q \geq p - 1 \\
(p-q, 0) & p \geq 1 \text{ and }  q \leq p - 1.
\end{cases}
$$
It is easy to check that $D$ is well-defined.  As we remarked already, however, $D$ is not continuous.  Now define
$$R = (\rho_2\times\rho_2)\circ D\colon  S(I_M, 2) \times S(I_N, 4) \to \{0\} \times  S(I_N, 2) \cup   I_M \times \{0\},$$
so that
$$
R(p, q) =
\begin{cases}
(0, \rho_2(q)) & p \leq 1 \\
(0, \rho_2(q - p+1)) & p \geq 1 \text{ and }  q \geq p - 1 \\
(\rho_2(p-q), 0) & p \geq 1 \text{ and }  q \leq p - 1.
\end{cases}$$
where $\rho_2 \colon S(I_M, 2) \to I_M$ and $\rho_2 \colon S(I_N, 4) \to S(I_N, 2)$ denote the projections from a subdivision back to the original:  $\rho_2(2i) = i$ and $\rho(2i+1) = i$, each $i \geq 0$.

Now we check that this map is continuous. For this, suppose that $(a, b)$ is a typical point in $S(I_M, 2) \times S(I_N, 4)$.  Write $S = S(I_M, 2) \times S(I_N, 4)$ and $A = I_M \times \{0\} \cup \{0\} \times S(I_N, 2)$.  Then we want to show that $R(x, y) \sim_A R(a, b)$, whenever $(x, y) \sim_S (a,b)$, that is, whenever $x = a, a\pm 1$ and $y = b, b \pm 1$.  In the following arguments, a key point is that, if $(x, y) \sim_S (a,b)$, then we have
$$b - a -2 \leq y - x \leq b - a +2.$$

\subsection{Case I: $b - a \geq 1$}  In this case, $(a, b)$ and all points in $S$ adjacent  to $(a, b)$  are mapped by $R$ to the axis $\{0\} \times S(I_N, 2)$ in $A$.

 If $b - a \geq 1$ and $a \geq 2$ (which entails $x \geq 1$ for any $(x, y)$ adjacent to $(a, b)$), then we have
 $$R(a, b) = \big(0, \rho_2(b-a+1)\big) \qquad \text{and} \qquad R(x, y) = \big(0, \rho_2(y-x+1)\big).$$
But $(x, y) \sim_S (a,b)$ entails
$$(b-a+1) -2 \leq y - x +1 \leq (b-a +1)+2,$$
and thus $\rho_2(b-a+1)-1 \leq \rho_2(y-x+1)  \leq \rho_2(b-a+1) + 1$.  It follows that, for $b-a \geq 1$ and  $a \geq 2$, we have $R(x, y) \sim_A R(a, b)$ whenever $(x, y) \sim_S (a,b)$.

 If $b - a \geq 1$ and $a \leq 1$ (which allows $x=2$ only if $a=1$, and otherwise entails $x \leq 1$),  then we have
 $$R(a, b) = \big(0, \rho_2(b)\big) \qquad \text{and} \qquad R(x, y) = \begin{cases} \big(0, \rho_2(y-x+1)\big) & \text{ if } a=1 \text{ and } x = 2\\
 \big(0, \rho_2(y)\big) & x \leq 1.\end{cases} $$
Here, then, we have either $R(x, y) = \big(0, \rho_2(y-1)\big)$ or  $R(x, y) = \big(0, \rho_2(y)\big)$.  But $(x, y) \sim_S (a,b)$ entails $b-1 \leq y \leq b+1$ and so $b-2 \leq y-1 \leq b$, and it follows that we have $\rho_2(b) - 1 \leq \rho_2(y-1) \leq \rho_2(y) \leq \rho_2(b) + 1$.  Here also  we have $R(x, y) \sim_A R(a, b)$.

\subsection{Case II: $-1 \leq b - a \leq 0$}  In this case, we have $R(a, b) = (0, 0)$.

For $(x, y) \sim_S (a,b)$ and $y-x \geq b-a$, we have $-1 \leq y - x \leq 2$.  Possible values for such $R(x, y)$ are $\{  \big(0, \rho_2(3)\big), \big(0, \rho_2(2)\big) , \big(0, \rho_2(1)\big), \big(0, \rho_2(0)\big), \big(\rho_2(1), 0\big) \} = \{ (0, 1), (0, 0), (1, 0) \}$. All these points are adjacent to  $R(a, b) = (0, 0)$ On the other hand, for $(x, y) \sim_S (a,b)$ and $y-x < b-a$, we have $-3 \leq y - x \leq -1$.   Possible values for such $R(x, y)$ are $\{   \big(\rho_2(1), 0\big),   \big(\rho_2(2), 0\big),   \big(\rho_2(3), 0\big)   \} = \{ (0, 0), (1, 0) \}$ and again all these points are adjacent to  $R(a, b) = (0, 0)$.

\subsection{Case III: $b - a \leq -2$}  In this case, $(a, b)$ and all points in $S$ adjacent  to $(a, b)$  are mapped by $R$ to the axis $I_M \times \{0\}$ in $A$.

Generally--if $b - a \leq -3$ or if $b-a = -2$ and $(x, y) \not= (a-1, b+1)$, in this case we have $R(a, b)$ and $R(x, y)$ given by
$$R(a, b) = ( \rho_2(a-b), 0) \qquad \text{and} \qquad R(x, y) = ( \rho_2(x-y), 0).$$
For $(x, y) \sim_S (a,b)$, we have
$$(b-a) -2 \leq y - x \leq (b-a)+2,$$
and thus $\rho_2(a-b)-1 \leq \rho_2(x-y)  \leq \rho_2(a-b) + 1$.  For such points, then, we have $R(x, y) \sim_A R(a, b)$.  If  $b-a = -2$, then we have $R(a, b) = (\rho_2(2), 0) = (1, 0)$.  But with $b-a = -2$,  the single adjacent point $(x, y) = (a-1, b+1)$ has (exceptionally, for Case III) $R(x, y) = \big(0, \rho_2(y-x+1)\big) = \big(0, \rho_2(1)\big) = (0,0)$. The remaining points adjacent to $(a, b)$ have $R(x, y) = (\rho_2(x-y), 0)$, and satisfy
$$(b-a) -2 \leq y - x \leq (b-a)+1,$$
hence $1 \leq x-y \leq 4$, and finally $0 \leq \rho_2(x-y) \leq 2$.  Here too, we have $R(x, y) \sim_A R(a, b)$.    This completes Case III, and with it we have shown that  $R$ is continuous.

It remains to observe that our $R\colon S(I_M, 2) \times S(I_N, 4) \to    I_M \times \{0\} \cup \{0\} \times S(I_N, 2)$ qualifies as a ``retraction" in  the sense given in \propref{prop: retraction iff cof}.  But this is immediate, since, from the definition of the function $D$ given above, we see  that $D$ fixes $S(\{0\}, 2) \times S(I_N, 4) \cup S(I_M, 2) \times \{0\} \subseteq S(I_M, 2) \times S(I_N, 4)$, and so $R$ restricts to $\rho_2 \times \rho_2$ here, as is required.  (Note that $S(\{0\}, 2) = \{0, 1\} \subseteq S(I_M, 2)$.)
\end{proof}

A  reflection on the details of the proof of \thmref{thm: well-pointed interval} together with  \exref{ex: non-cofibration} will reveal that \defref{def: digital cofibration} abstracts exactly the kind of ``homotopy extension property" that an inclusion $j\colon \{0\} \to I_M$ possesses, in the digital setting. Namely, we must allow for a subdivision of the domain as well as longer paths in the range, before a given homotopy may be extended.

Some of the basic properties of cofibrations carry over from the topological to the digital setting.  For instance, we have the following consequence of \propref{prop: retraction iff cof}.

\begin{lemma}\label{lem: 1 x i cofibration}
If an inclusion $j\colon A \to X$ of digital images is a cofibration, then so is the inclusion $\mathrm{id}\times j \colon Z \times A \to Z \times X$ for any digital image $Z$.
\end{lemma}

\begin{proof}
Since $j\colon A \to X$ is a cofibration, \propref{prop: retraction iff cof} gives, for each $I_N$, subdivisions  $S(X, k)$ and $S(I_N, l)$ with $l \geq 2$, and a ``retraction"
$$R\colon S(X, k) \times S(I_N, lm) \to X \times \{0\} \cup A \times S(I_N, l)$$
as in that statement.  But then
$$\rho_k \times R \colon S(Z, k) \times S(X, k) \times S(I_N, lm) \to Z \times X \times \{0\} \cup Z \times A \times S(I_N, l)$$
is a retraction in the same sense, corresponding to the inclusion $\text{id}\times j\colon Z \times A \to Z \times X$.  Hence, again by \propref{prop: retraction iff cof}, this map is also a cofibration.
\end{proof}

For example, this result, combined with \thmref{thm: well-pointed interval}, implies that the inclusion $j\colon I^{n-1}  \to I^{n}$ of a face of the $n$-cube, for any $n$, is a cofibration.

On the other hand, not all properties of cofibrations carry over.  For example, the usual (and easy) argument that shows a composition of cofibrations is again a cofibration in the topological setting breaks down here.  We are unsure whether or not, according to our definition, a composition of cofibrations is always a cofibration.

We establish another basic example of a cofibration.  As we will see in the next section, this example and the previous one lead to important examples of what might be called fibrations in  the digital setting.

\begin{theorem}\label{thm: endpoint cofibration}
For any $M$, the inclusion $j\colon \{0, M\} \to I_M$ is  a  cofibration.
\end{theorem}

\begin{proof}
As in \thmref{thm: well-pointed interval}, we will apply \propref{prop: retraction iff cof}. For this we seek, for each $I_N$,  subdivisions and a retraction (in the sense of \propref{prop: retraction iff cof})
$$R \colon S(I_M, k) \times S(I_N, lm) \to    I_M \times \{0\} \cup \{0, M\} \times S(I_N, l)$$
with $l \geq 2$.  It is sufficient to use $l = m = 2$, but as we will see, we will generally need to allow for a larger $k$.  We will construct a suitable
$$R \colon S(I_M, k) \times S(I_N, 4) \to    I_M \times \{0\} \cup \{0, M\} \times S(I_N, 2).$$
Notice that such a map may be viewed as a map $R \colon I_{kM+k-1} \times I_{4N+3} \to I_M \times \{0\} \cup \{0, 1\} \times I_{2N+1}$, so visually we want to retract a rectangle onto its (contracted) left, bottom, and right edges.   In the topological setting, a retraction of $I \times I$ onto its left, bottom, and right edges is achieved by centrally projecting from a point such as $(0.5, 1.5)$.  In the digital setting, we may use an analogous approach, but we need to adapt considerably to ensure continuity and also that the  technical requirement of  \propref{prop: retraction iff cof} is satisfied.

As a first step, consider a rectangle $I_{4K+3} \times I_K$ for some $K \geq 1$ (typically $K$ will be much larger than $1$).  We begin by describing a continuous map
$$R_K \colon I_{4K+3} \times I_K \to I_{2K+1} \times \{0\} \cup \{0, 1\} \times I_{\rho_2(K)}$$
where, as usual, $\rho_2(K) = \lfloor K/2 \rfloor$.  We divide the rectangle $I_{4K+3} \times I_K$ into symmetric left-hand and right-hand halves: $I_{2K+1} \times I_K$ and $[2K+2, 4K+3] \times I_K$.  We will describe $R_K$ on the left-hand half, and check that it is continuous there, and then use symmetry to conclude the same for the right-hand half, and hence the whole rectangle.  To this end, divide the left-hand half $I_{2K+1} \times I_K$ into a lower-left trapezoid ($T_1$), an upper-right triangle ($T_2$), and a vertical interval, as follows:
$$I_{2K+1} \times I_K  = T_1 \cup T_2 \cup \{ (2K+1, j) \mid 0 \leq j \leq K \},$$
with
$$\begin{aligned}
T_1 & = \{ (i, j) \in I_{2K+1} \times I_K \mid 0 \leq j \leq K \text{ and } 0 \leq i \leq 2K - j\}, \\
T_2 & = \{ (i, j) \in I_{2K+1} \times I_K \mid 1 \leq j \leq K \text{ and } 2K - j < i \leq 2K\}.
\end{aligned}
$$
Now define $R_K$ on $T_1$ using the same formulas we used in the proof of \thmref{thm: well-pointed interval}, namely, for $0 \leq j \leq K$, define
$$
R_K(i, j) =
\begin{cases}
(0, \rho_2(j)) & i \leq 1 \\
(0, \rho_2(j - i+1)) & 1 \leq  i \leq j + 1 \\
(\rho_2(i-j), 0) &   j + 1 \leq i \leq 2K - j.
\end{cases}$$
On the interval $\{ (2K+1, j) \mid 0 \leq j \leq K \}$, define $R_K$ as
$$R_K(2K+1, j) = (\rho_2(2K+1), 0) = (K, 0),$$
for each $j = 0, \ldots, K$.  Finally,  on the triangle $T_2$, define
$$R_K(i, j) = (i-K, 0),$$
for each $i = j+1, \ldots, 2K$.

We check that this gives a continuous map.  Consider a typical point $(a, b) \in  I_{2K+1} \times I_K$.  We must show that $R_K(x, y) \sim R_K(a, b)$ for each $(x, y)$ with $(x, y) \sim (a, b)$.  If $a+b \leq 2K-2$, then $(a, b)$ and all points adjacent to $(a,b)$ are in $T_1$, and from the proof of  \thmref{thm: well-pointed interval} we know that the formulas used here to define $R_K$ give a continuous map.  Also, if $a+b \geq 2K+2$ and $a \leq 2K-1$, then $(a, b)$ and all points adjacent to $(a,b)$ are in $T_2$.   Here, it is clear that $R_K$ preserves adjacency, since if $(x, y) \sim (a, b)$, then $x \sim a$ and hence $R_K(x, y) = x-K \sim a - K = R_K(a, b)$.  If $a+b \geq 2K+2$ and $a = 2K$, then the previous remark  plus the fact that $R_K(a+1, b) = R_K(a+1, b\pm 1) = (K, 0) =  R_K(a, b)$, shows  that $R_K(x, y) \sim R_K(a, b)$ when $(x, y) \sim (a, b)$.

 If $(a, b)$ is such that $2K-1 \leq a+b \leq 2K+1$, then we have points $(x, y)$ adjacent to $(a, b)$  in both $T_1$ and $T_2$.  Suppose that we have $a+b = 2K$, so that $(a, b) \in T_1$.  Furthermore, suppose that $K+1 \leq a \leq 2K-1$.  For points adjacent to such an $(a, b)$ and in $T_1$, adjacency is preserved by $R_K$, as we have already observed.  The only points adjacent to such an $(a, b)$ and not in $T_1$ are the three points $(a, b+1)$, $(a+1, b)$, and $(a+1, b+1)$.  But for such an $(a, b)$ we have
$$R_K(a, b) = (\rho_2(a-b), 0) =  (\rho_2\big(a-(2K - a)\big), 0) =    (\rho_2\big(2(a-K)\big), 0) = (a-K, 0),$$
whilst for the three adjacent points in $T_1$ we have
$$R_K(a, b+1) =  (a-K, 0), \quad R_K(a+1, b) =  (a+1-K, 0), \quad R_K(a+1, b+1) =  (a+1-K, 0).$$
Since $(a-K, 0) \sim (a+1-K, 0)$ in $I_{2K+1}$,  it follows that $R_K$ preserves adjacency when  $a+b = 2K$ and $K+1 \leq a \leq 2K-1$.  When  $a+b = 2K$ and $a$ equals either $K$ or $2K$, and also when $a+b$ equals either $2K-1$ or $2K+1$, a minor variation on this argument shows that $R_K$ preserves adjacency in all these cases, too.  We omit these details.

Thus far, we have argued that, for $(a, b) \in I_{2K} \times I_K$, we have $R_K(x, y) \sim R_K(a, b)$ whenever $(x, y) \sim (a, b)$.   Recall that we have defined $R_K(2K+1, j) = (K, 0) = R_K(2K, j)$, for each $j = 0, \ldots, K$.  We will extend the definition of $R_K$ to $[2K+2, 4K+3] \times I_K$ in such a way that we also have $R_K(2K+2, j) = (K, 0)$, for each $j = 0, \ldots, K$.   When that is done, clearly we will have $R_K(x, y) \sim R_K(2K+1, b)$ whenever $(x, y) \sim (2K+1, b)$.

We extend $R_K$ to $[2K+2, 4K+3] \times I_K$ by first reflecting $[2K+2, 4K+3] \times I_K$ in the vertical line $y = 2K + 1.5$, applying $R_K$ as we have defined it on  $I_{2K+1} \times I_K$ (which contracts $I_{2K+1} \times I_K$ to the left and bottom edges of $I_K \times I_{\rho_2(K)}$) and then reflecting back in the vertical line $y = \rho_2(K) + 0.5$.  The reflections obviously preserve adjacency, and so this gives a map that is continuous at least on $[2K+3, 4K+3] \times I_K$.  Notice  that this definition gives  $R_K(2K+2, j) = (K, 0)$, for each $j = 0, \ldots, K$, and so the previous remarks show that we have defined a continuous map
$$R_K \colon I_{4K+3} \times I_K \to I_{2K+1} \times \{0\} \cup \{0, 1\} \times I_{\rho_2(K)}$$
as desired.

Next, we may restrict this map to any rectangle that is just as wide, but not so tall.  That is, suppose we have a $K'$ with $K' \leq K$.  Then the $R_K$ we just defined restricts to give a continuous map
$$R_K \colon I_{4K+3} \times I_K' \to I_{2K+1} \times \{0\} \cup \{0, 1\} \times I_{\rho_2(K')}.$$
In particular, if $K' = 4N+3$ for some $N$, with $4N+3 \leq K$, then we have a restriction of $R_K$ to a continuous map
$$R_K \colon I_{4K+3} \times I_{4N+3} \to I_{2K+1} \times \{0\} \cup \{0, 1\} \times I_{\rho_2(4N+3)}.$$
Furthermore, we may identify $I_{4K+3} = S(I_{2K+1}, 2)$, $I_{4N+3} = S(I_{N}, 4)$, and $I_{\rho_2(4N+3)} = I_{2N+1} = S(I_N, 2)$.  With these identifications, then, we have a continuous map
$$R_K \colon S(I_{2K+1}, 2) \times S(I_{N}, 4) \to I_{2K+1} \times \{0\} \cup \{0, 1\} \times S(I_N, 2).$$
A review of the way in which we defined $R_K$  reveals that, when restricted to   $S(I_{2K+1}, 2) \times \{0\} \cup S( \{0, 2K+1\}, 2) \times S(I_{N}, 4)$, we have
$$R_K = \text{id}\times \rho_2 = \rho_2 \times \rho_2\colon S(I_{2K+1}, 2) \times \{0\} \to I_{2K+1},$$
$$R_K = \rho_2 \times \rho_2\colon \{0, 1\} \times S(I_{N}, 4)  \to \{0\} \times  S(I_N, 2), $$
and
$$R_K = \rho_2 \times \rho_2\colon \{4K+2, 4K+3\} \times S(I_{N}, 4) \to \{2K+1\}  \times  S(I_N, 2).$$
Note that, in the above expressions, we have $S( \{0, 2K+1\}, 2) = \{0, 1\} \cup \{4K+2, 4K+3\}$.
So, for any $N$ with $4N+3 \leq K$, we have a commutative diagram as follows:
$$\xymatrix{ S( \{0, 2K+1\}, 2)  \times S(I_N,  4) \cup S(I_{2K+1}, 2)  \times \{0\} \ar[r]^-{\text{incl.}} \ar[rd]_-{\rho_2 \times \rho_2} & S(I_{2K+1}, 2)  \times S(I_N, 4) \ar[d]^{R_K}\\
 & I_{2K+1} \times \{0\} \cup \{0, 2K+1\} \times S(I_N, 2) }$$

The final step is to take a general $I_M \times I_N$, and fit it into the above.  For this, we subdivide $I_M$ by a suitable power of $2$.  For any $p \geq 2$, and any $M \geq 1$, observe that we have
$$S(I_M, 2^p) = S(  S(I_M, 2^{p-2}), 4) = I_{4K+3}$$
with $K = M 2^{p-2} + 2^{p-2} - 1$.  So, given an $M$ and an $N$,  choose a $p$ for which we have   $4N+3 \leq M 2^{p-2} + 2^{p-2} - 1$ (the smallest such $p$ will do).  Then from the above, with $K = M 2^{p-2} + 2^{p-2} - 1$, we have a map $R_K$ and a  commutative diagram as above.  But here, we have $S(I_M, 2^{p-1}) =S(  S(I_M, 2^{p-2}), 2) = I_{2K+1}$, and so we may project with $\rho_{2^{p-1}}\colon I_{2K+1} = S(I_M, 2^{p-1}) \to I_M$, to obtain a commutative diagram (still with $K = M 2^{p-2} + 2^{p-2} - 1$)
$$\xymatrix{ S( \{0, 2K+1\}, 2)  \times S(I_N,  4) \cup S(I_{M}, 2^p)  \times \{0\} \ar[r]^-{\text{incl.}} \ar[rd]_-{\rho_{2^{p}} \times \rho_2} & S(I_{M}, 2^p)  \times S(I_N, 4) \ar[d]^{\rho_{2^{p-1}}\circ R_K}\\
 & I_{M} \times \{0\} \cup \{0, M\} \times S(I_N, 2) }$$
Then $\rho_{2^{p-1}}\circ R_K$ is a retraction in the sense required in \propref{prop: retraction iff cof}.
\end{proof}

\section{Digital Fibrations}\label{sec: fibrations}

Now we use our results on digital cofibrations to develop some ideas about fibrations in the digital setting. 
Despite the heading of the section, however,   our development stops short of offering a general definition of fibration in the digital setting: we have been unable, so far, to formulate a general definition that includes the examples we focus on here, and that also has some use beyond them.
Rather, we focus on
developing an adapted homotopy lifting property for the evaluation maps (path ``fibrations")  of \defref{def:dig eval maps} as well as the based version of one of these  (see \defref{def: based path fibration} below). Our reasons for this focus are two-fold.  First, we wish to build on the results of \secref{sec: cofibrations} so as to add depth to our development.  Second, these evaluation maps, in the topological setting, are germane to the topics of Lusternik--Schnirelmann category (mentioned at several points in the introduction) and a second, related, numerical invariant called \emph{topological complexity} (see \cite{Far06} and \cite{G-L-V18}).  In fact, we use one of the results developed in this section to give a preliminary treatment of Lusternik-Schnirelmann category in the digital setting in \secref{sec: cat and TC} below.  We do not attempt to treat topological complexity in this paper.  But the results of this section do provide a basis for just such a treatment, which we intend to pursue in a subsequent paper.  See also \secref{sec: future} below for some more discussion of these topics.

 In the topological setting, a \emph{fibration}  is a map that has the homotopy lifting property.    That is, $p\colon E \to B$ is a fibration when, for any $Z$ the following  commutative diagram has a filler $\bar{H}\colon Z \times I \to E$.  Here, the map $i\colon \{0\} \to I$ denotes inclusion of the endpoint $0$ into the unit interval.  Namely, the homotopy $H\colon Z \times I \to B$ lifts through $p$ to a homotopy $\bar{H}\colon Z \times I \to E$ that begins at the map $f \colon Z \to E$:
$$\xymatrix{
Z \times \{0\} \ar[r]^-{f} \ar[d]_-{\mathrm{id}_Z \times i}  & E \ar[d]^-{p}\\
Z \times I \ar[r]_-{H}  \ar@{.>}[ru]_{\bar{H}} & B   }
$$
Furthermore, cofibrations provide an important source of fibrations, because of the following result (sometimes referred to as Borsuk's theorem).

\begin{theorem}\label{thm: Borsuk}
In the topological setting, suppose that we have an inclusion $j \colon A \to X$ of a closed subspace $A$ into $X$.  If $j$ is a cofibration, then the induced map of mapping spaces
$j^*\colon \map(X, Y) \to \map(A, Y)$ is a fibration, for any $Y$.
\end{theorem}

\noindent{}This result---still in the topological setting---is then used to deduce various evaluation maps, such as $PY \to Y$ and its based counterpart are fibrations. We now adapt the same line of development into the digital setting.

Just as we saw for cofibrations, if we simply repeat the ordinary definition of fibration in the digital setting, many interesting examples are excluded from qualifying as a fibration.  Instead, we take our cue from \thmref{thm: Borsuk}, and develop in this section an adapted homotopy lifting property for certain path fibrations. 
We begin with a digital version of \thmref{thm: Borsuk}.

\begin{theorem}\label{thm: Digital Borsuk}
Let $j \colon A \to X$ be an inclusion of digital images.   For any digital images $Z$ and $Y$, suppose we are given a commutative diagram
$$\xymatrix{
Z \times \{0\} \ar[r]^-{f} \ar[d]_-{\mathrm{id}_Z \times i}  & \map(X, Y) \ar[d]^-{j^*}\\
Z \times I_M \ar[r]_-{H}  & \map(A, Y).   } $$
If $j$ is a cofibration, then there are subdivisions $S( -, k)$ and $S(I_M, l)$, and  a filler $\overline{H}\colon S(Z, k) \times S(I_M, l) \to \map( S(X, k), Y)$ in the following commutative diagram:
\begin{equation}\label{eq: fibration diag}
\xymatrix{
S(Z, k) \times \{0\} \ar[d]_{\mathrm{id}\times i} \ar[r]^-{\rho_k\times \mathrm{id}} & Z \times \{0\} \ar[r]^-{f} \ar[d]_-{\mathrm{id}\times i}  &  \map(X, Y) \ar[d]^-{j^*} \ar[r]^-{ (\rho_k)^* } &  \map( S(X, k), Y) \ar[d]^{(S(j))^*}\\
S(Z, k) \times S(I_M, l)\ar[r]_-{\rho_k \times \rho_l}  \ar@{.>}[rrru]^-{\overline{H}} &Z \times I_M \ar[r]_-{H}   &  \map(A, Y)  \ar[r]_-{ (\rho_k)^*  } & \map( S(A, k), Y)   }
\end{equation}
\end{theorem}

\begin{proof}
Begin with the given data $f$ and $H$, adjoint each to give maps
$$\widehat{f} \colon Z \times X \times \{0\} \to Y \qquad \text{and} \qquad \widehat{H} \colon Z \times A \times I_M \to Y,$$
and adjoint once more to get continuous maps
$$f' \colon Z \times X  \to \map(\{0\}, Y) \qquad \text{and} \qquad H' \colon Z \times A  \to \map(I_M, Y).$$
Continuity is preserved at each step, by \propref{prop: exponential law}.   Both steps may be combined into the formulas
$$f'(z, x)(0) = f(z, 0)(x) \qquad \text{and} \qquad H'(z, a)(t) = H(z, t)(a),$$
for typical points $a \in A$, $x \in X$, $z \in Z$, and $t \in I_M$.  One checks from these formulas that we have a  commutative diagram
$$\xymatrix{
Z \times A\ar[r]^-{H'} \ar[d]_-{\mathrm{id} \times j}  & \map(I_M, Y) \ar[d]^-{i^*}\\
Z \times X \ar[r]_-{f'}  & \map(\{0\}, Y).   } $$
Now, as observed in  item (2) of \exref{ex: induced maps},  $i^* \colon \map(I_M, Y) \to \map(\{0\}, Y)$ is nothing other than the evaluation map $\text{ev}_0 \colon P_MY \to Y$, and $\mathrm{id} \times j \colon Z \times A \to Z \times X$ is a cofibration, by \lemref{lem: 1 x i cofibration}.  Therefore, \emph{per} \defref{def: digital cofibration}, there are subdivisions $S( Z \times X, k) = S( Z, k)  \times S( X, k)$ and $S(I_M, l)$, and  a filler $\overline{H'} \colon S( Z, k)  \times S( X, k) \to \map( S(I_M, l), Y)$ in the following commutative diagram:
\begin{equation}\label{eq: cof diag}
\xymatrix{
S( Z, k)  \times S(A, k) \ar[d]_{\text{id} \times S(j,k)} \ar[r]^-{\rho_k\times \rho_k} & Z \times A \ar[r]^-{H} \ar[d]_-{\mathrm{id} \times j }  & \map(I_M, Y) \ar[d]^-{\mathrm{ev}_0} \ar[r]^-{ (\rho_l)^* } &  \map( S(I_M, l), Y) \ar[ld]^{\mathrm{ev}_0}\\
S( Z, k)  \times S(X, k) \ar[r]_-{\rho_k\times \rho_k}  \ar@{.>}[rrru]^-{\overline{H'}} &Z \times X \ar[r]_-{f'}   & \map(\{0\}, Y)  }
\end{equation}
The adjoint of the filler $\overline{H'}$ gives a map $S( Z, k)  \times S(I_M, l) \times S( X, k) \to Y$, and a second adjoint finally gives a map
$$\overline{H}\colon S( Z, k)  \times S(I_M, l) \to  \map(S( X, k), Y),$$
defined by the formula $\overline{H}(z', t')(x') = \overline{H'}(z', x')(t')$, for typical points $z' \in S( Z, k)$, $x' \in S( X, k)$,  and $t' \in S(I_M, l)$.  We now check that this map provides the desired filler for  diagram (\ref{eq: fibration diag}) in the enunciation.

First consider the lower right triangle of (\ref{eq: fibration diag}).  For typical points $z' \in S( Z, k)$, $a' \in S( A, k)$,  and $t' \in S(I_M, l)$, we have
$$\begin{aligned}
\big( S(j, k)\big)^*\circ \overline{H} (z', t')(a') &= \overline{H} (z', t') \big( S(j, k)(a') \big)\\
&= \overline{H'} (z', S(j, k)(a') )(t').
\end{aligned}
$$
Now from diagram (\ref{eq: cof diag}) above,  we may continue this string of equalities, to write
$$\begin{aligned}
\big( S(j, k)\big)^*\circ \overline{H} (z', t')(a') &=  \overline{H'} (z', S(j, k)(a') )(t')\\
&= (\rho_l)^*\circ H'\circ (\rho_k\times\rho_k)(z', a')(t')\\
&= H'\big( \rho_k(z'), \rho_k(a')\big) \big(\rho_l(t')\big)\\
&= H \big( \rho_k(z'), \rho_l(t')\big) \big(\rho_k(a')\big)\\
&= (\rho_k)^*\circ H \circ (\rho_k\times\rho_l)(z', t')(a').
\end{aligned}
$$
That is, we have
$$\big( S(j, k)\big)^*\circ \overline{H} =  (\rho_k)^*\circ H \circ (\rho_k\times\rho_l)\colon S(Z, k) \times S(I_M, l) \to \map( S(A, k), Y).$$

Next consider the upper left triangle of (\ref{eq: fibration diag}).  For typical points $z' \in S( Z, k)$, $x' \in S( X, k)$, we have
$$\begin{aligned}
\overline{H}\circ (\text{id} \times i) (z', 0)(x') &=  \overline{H} (z', 0 )(x') =  \overline{H'} (z', x' )(0)\\
&= f'\big( \rho_k(z'), \rho_k(x')\big) (0) \text{ \ \ \   (from diagram (\ref{eq: cof diag}))}\\
&= f \big( \rho_k(z'), 0\big) \big(\rho_k(x')\big) = (\rho_k)^*\circ f \circ (\rho_k\times i)(z', 0)(x').
\end{aligned}
$$
So we have
$$ \overline{H} \circ (\text{id} \times i) = (\rho_k)^*\circ f \circ (\rho_k\times i) \colon S(Z, k) \times \{0\} \to \map( S(X, k), Y),$$
and $\overline{H}$ is  the desired filler.
\end{proof}

We may deduce from this result the adapted homotopy lifting property possessed by the map $\text{ev}_0\colon P_NY \to Y$ that qualifies it to be called a \emph{path fibration} in the digital setting.
In this result, we use $i$ to denote a generic  $i\colon \{0\} \to I_M$, for any $M$,  and $j\colon \{0\} \to I_N$ to emphasize the cofibration of \thmref{thm: well-pointed interval} (both are cofibrations, however).  We also make the identification of  $\text{ev}_0\colon P_NY \to Y$ and $j^* \colon \map(I_N, Y) \to \map(\{0\}, Y)$   observed in  item (2) of \exref{ex: induced maps} and already used in the proof of \thmref{thm: Digital Borsuk}.

\begin{corollary}\label{cor: path fibration}
For any digital image $Y$, suppose given a commutative diagram
$$\xymatrix{
Z \times \{0\} \ar[r]^-{f} \ar[d]_-{\mathrm{id} \times i}  & P_NY\ar[d]^-{\text{ev}_0}\\
Z \times I_M \ar[r]_-{H}  & Y.   } $$
Then there are subdivisions $S( Z, k)$, $S(I_N, k) = I_{kN+k-1}$, and $S(I_M, l)= I_{lM+l-1}$, and  a filler $\overline{H}\colon S(Z, k) \times I_{lM+l-1} \to P_{kN+k-1}Y$ in the following commutative diagram:
$$
\xymatrix{
S(Z, k) \times \{0\} \ar[d]_{\mathrm{id}\times i} \ar[r]^-{\rho_k\times \mathrm{id}} & Z \times \{0\} \ar[r]^-{f} \ar[d]_-{\mathrm{id}\times i}  &  P_NY \ar[d]^-{\text{ev}_0} \ar[r]^-{ (\rho_k)^* } &  P_{kN+k-1}Y \ar[d]^{\text{ev}_0}\\
S(Z, k) \times I_{lM+l-1}\ar[r]_-{\rho_k \times \rho_l}  \ar@{.>}[rrru]^-{\overline{H}} &Z \times I_M \ar[r]_-{H}   &  Y  \ar@{=}[r] & Y  }
$$
\end{corollary}

\begin{proof}
By combining \thmref{thm: well-pointed interval}  and \thmref{thm: Digital Borsuk}, we obtain a filler in the following diagram:
$$
\xymatrix{
S(Z, k) \times \{0\} \ar[d]_{\mathrm{id}\times i} \ar[r]^-{\rho_k\times \mathrm{id}} & Z \times \{0\} \ar[r]^-{f} \ar[d]_-{\mathrm{id}\times i}  &  P_NY \ar[d]^-{\text{ev}_0} \ar[r]^-{ (\rho_k)^* } &  \map( S(I_N, k), Y) \ar[d]^{(S(j))^*}\\
S(Z, k) \times S(I_M, l)\ar[r]_-{\rho_k \times \rho_l}  \ar@{.>}[rrru]^-{\overline{H}} &Z \times I_M \ar[r]_-{H}   &  Y  \ar[r]_-{ (\rho_k)^*  } & \map( S(\{0\}, k), Y)   }
$$
Now here, we have a right inverse for the map $\rho_k \colon S(\{0\}, k) \to \{0\}$.  Namely, writing $S(\{0\}, k)$ as $I_{k-1}$, the inclusion $i \colon \{0\} \to I_{k-1}$ satisfies
$$\rho_k\circ i = \text{id}\colon \{0\} \to S(\{0\}, k) \to \{0\}.$$
(For a general cofibration $j \colon A \to X$, we usually do not have a map $A \to S(A, k)$.)   Therefore, $i^*\colon \map( S(\{0\}, k), Y) \to \map( \{0\}, Y) = Y$ is a left inverse of $ (\rho_k)^*\colon Y \to \map( S(\{0\}, k), Y)$.  Adding this to the right-hand part of the diagram, and re-writing $S(I_M, l)= I_{lM+l-1}$, and $\map( S(I_N, k), Y) = P_{kN+k-1}Y$, we obtain the following:
$$
\xymatrix{
S(Z, k) \times \{0\} \ar[d]_{\mathrm{id}\times i} \ar[r]^-{\rho_k\times \mathrm{id}} & Z \times \{0\} \ar[r]^-{f} \ar[d]_-{\mathrm{id}\times i}  &  P_NY \ar[d]^-{\text{ev}_0} \ar[r]^-{ (\rho_k)^* } &  P_{kN+k-1}Y \ar[d]^{i^*\circ(S(j))^*}\\
S(Z, k) \times I_{lM+l-1}\ar[r]_-{\rho_k \times \rho_l}  \ar@{.>}[rrru]^-{\overline{H}} &Z \times I_M \ar[r]_-{H}   &  Y  \ar[r]_-{ i^*\circ(\rho_k)^*=\text{id}  } & Y  }
$$
But observe that, for the right-hand vertical map, we have $i^*\circ(S(j))^* = (S(j)\circ i)^*$, and $S(j)\circ i\colon \{0\} \to S(\{0\}, k) = I_{k-1} \to S(I_N, k) =  I_{kN+k-1}$ is simply $i\colon \{0\} \to I_{kN+k-1}$.  Thus we may identify this right-hand vertical map with $\text{ev}_0\colon  P_{kN+k-1}Y \to Y$.  This results in the desired diagram.
\end{proof}

\begin{remark}\label{rem:TC fibration induced}
Recall from \defref{def:dig eval maps} the evaluation map $\pi\colon P_N Y \to Y \times Y$.  If $N \geq 2$, then we may identify  $\pi$ with the map induced on mapping spaces $j^* \colon \map( I_N, Y) \to \map(\{0, N\}, Y)$ by  the cofibration $j\colon \{0, N\} \to I_N$ of \thmref{thm: endpoint cofibration}. Note that we need $N \geq 2$ so that there are no adjacency constraints on where the two points $0$ and $N$ may be mapped.
\end{remark}

The identification of the preceding remark, together with \thmref{thm: Digital Borsuk}, leads to the following adapted homotopy lifting property for the evaluation map $\pi$, thus qualifying it also to be called a path fibration.   In this result, we make various identifications similar to those made in \corref{cor: path fibration}.

\begin{corollary}\label{cor: TC fibration}
For any digital space $Y$, suppose given a commutative diagram
$$\xymatrix{
Z \times \{0\} \ar[r]^-{f} \ar[d]_-{\mathrm{id} \times i}  & P_NY\ar[d]^-{\pi}\\
Z \times I_M \ar[r]_-{H}  & Y\times Y   } $$
with $N \geq 2$.
Then there are subdivisions $S( Z, k)$, $S(I_N, k) = I_{kN+k-1}$, and $S(I_M, l)= I_{lM+l-1}$, and  a filler $\overline{H}\colon S(Z, k) \times I_{lM+l-1} \to P_{kN+k-1}Y$ in the following commutative diagram:
$$
\xymatrix{
S(Z, k) \times \{0\} \ar[d]_{\mathrm{id}\times i} \ar[r]^-{\rho_k\times \mathrm{id}} & Z \times \{0\} \ar[r]^-{f} \ar[d]_-{\mathrm{id}\times i}  &  P_NY \ar[d]^-{\pi} \ar[r]^-{ (\rho_k)^* } &  P_{kN+k-1}Y \ar[d]^{\pi}\\
S(Z, k) \times I_{kN+k-1}\ar[r]_-{\rho_k \times \rho_l}  \ar@{.>}[rrru]^-{\overline{H}} &Z \times I_M \ar[r]_-{H}   &  Y \times Y  \ar@{=}[r]  & Y \times Y   }
$$
\end{corollary}

\begin{proof}
We use an argument similar to that of \corref{cor: path fibration}. Combine  \thmref{thm: endpoint cofibration}  and \thmref{thm: Digital Borsuk} to obtain a commutative diagram as follows.
$$
\xymatrix{
S(Z, k) \times \{0\} \ar[d]_{\mathrm{id}\times i} \ar[r]^-{\rho_k\times \mathrm{id}} & Z \times \{0\} \ar[r]^-{f} \ar[d]_-{\mathrm{id}\times i}  &  P_NY \ar[d]^-{\pi} \ar[r]^-{ (\rho_k)^* } &  P_{kN+k-1}Y \ar[d]^{(S(j))^*}\\
S(Z, k) \times I_{kN+k-1}\ar[r]_-{\rho_k \times \rho_l}  \ar@{.>}[rrru]^-{\overline{H}} &Z \times I_M \ar[r]_-{H}   &  Y \times Y \ar[r]_-{ (\rho_k)^*  } & \map( S(\{0, N\}, k), Y)   }
$$
Here we have identified $S(I_N, k) = I_{kN+k-1}$, $\map( S(I_N, k), Y) = \map( I_{kN+k-1}, Y) = P_{kN+k-1} Y$, and written $j^* \colon \map( I_N, Y) \to \map(\{0, N\}, Y)$ as $\pi\colon P_N Y \to Y \times Y$, using the cofibration $j\colon \{0, N\} \to I_N$.  Now observe that, as in \corref{cor: path fibration}, we have a right inverse to the map $\rho_k \colon S( \{0, N\}, k) \to \{0, N\}$.  Namely, the map $i \colon  \{0, N\} \to S( \{0, N\}, k) = I_{kN+k-1}$ given by $i(0) = 0$ and $i(N) = kN+k-1$ satisfies
$$\rho_k\circ i = \text{id}\colon  \{0, N\} \to S( \{0, N\}, k) \to  \{0, N\}.$$
Hence, $i^*\colon \map( S(\{0, N\}, k), Y) \to \map(\{0, N\}, Y) = Y\times Y$ is a left inverse of $ (\rho_k)^*\colon Y\times Y \to \map( S(\{0, N\}, k), Y)$. Composing the bottom right horizontal and the right-hand vertical maps with $i^*$ results in the following diagram:
$$
\xymatrix{
S(Z, k) \times \{0\} \ar[d]_{\mathrm{id}\times i} \ar[r]^-{\rho_k\times \mathrm{id}} & Z \times \{0\} \ar[r]^-{f} \ar[d]_-{\mathrm{id}\times i}  &  P_NY \ar[d]^-{\pi} \ar[r]^-{ (\rho_k)^* } &  P_{kN+k-1}Y \ar[d]^{i^*\circ(S(j))^*}\\
S(Z, k) \times I_{kN+k-1}\ar[r]_-{\rho_k \times \rho_l}  \ar@{.>}[rrru]^-{\overline{H}} &Z \times I_M \ar[r]_-{H}   &  Y \times Y \ar[r]_-{ i^*\circ(\rho_k)^*=\text{id} } & Y \times Y  }
$$
Finally, observe that $i^*\circ(S(j))^* = (S(j)\circ i)^*$, and $S(j)\circ i\colon \{0, N\}  \to S(\{0, N\} , k) = I_{k-1} \cup \{kN, kN+1, \ldots, kN + k-1\} \to S(I_N, k) =  I_{kN+k-1}$ is simply the map that sends $0 \mapsto 0$ and $N \mapsto kN+k-1$.    So we may identify this right-hand vertical map with $\pi\colon  P_{kN+k-1}Y \to Y\times Y$.
\end{proof}

\begin{remark}\label{rem: not surjective}
In \corref{cor: path fibration}, the evaluation map $\text{ev}_0\colon P_NY \to Y$ is always surjective.  In \corref{cor: TC fibration}, however, the evaluation map $\pi\colon P_N Y \to Y \times Y$ in general is not surjective.  This is because there may be points $a, b \in Y$ ``too far apart" to be connected by a path in $Y$ of length $N$.  Notice, though, that so long as $Y$ is connected,  the subdivided counterpart of $\pi$, namely  $(S(j))^* \colon \map( S(I_N, k), Y) \to  \map( S(\{0, N\}, k), Y\times Y) $, will always be surjective if $k$ is sufficiently large.
\end{remark}

\thmref{thm: Digital Borsuk} also leads to a corresponding result about induced maps of \emph{based mapping spaces}, which is a further source of important examples of fibrations.  In the topological setting, we have a more general result that says the restriction of a fibration is again a fibration.  Since we have no general notion of fibration, as yet, in the digital setting, we will restrict ourselves to this particular situation.

Suppose that $j\colon A \to X$  is an inclusion of \emph{based} digital images, which is to say that we specify a basepoint $a_0 \in A \subseteq X$, and $j(a_0) = a_0 \in X$ is the basepoint of $X$.  Furthermore, suppose that $Y$ is a based digital image with basepoint $y_0 \in Y$, and let $\map_*(X, Y)$, respectively $\map_*(A, Y)$, denote the based mapping spaces that consist of continuous maps $f\colon X \to Y$, respectively $f\colon A \to Y$,  with $f(a_0) = y_0$.  Then we have an induced map of based mapping spaces $j^*\colon  \map_*(X, Y) \to \map_*(A, Y)$.  Furthermore, if $x_0 \in X$ is a choice of basepoint in a digital image $X$, then we may regard $kx_0$ as a basepoint in $S(X, k)$ (its coordinates will each be scaled by $k$, according to our description of subdivision) and with this convention the canonical map $\rho_k\colon  S(X, k) \to X$ is a based map.

\begin{theorem}\label{thm: Based Digital Borsuk}
With the notation above, let $j \colon A \to X$ be a based inclusion of based digital images.   For any digital image $Z$, and any based digital image $Y$, suppose we are given a commutative diagram
$$\xymatrix{
Z \times \{0\} \ar[r]^-{f} \ar[d]_-{\mathrm{id}_Z \times i}  & \map_*(X, Y) \ar[d]^-{j^*}\\
Z \times I_m \ar[r]_-{H}  & \map_*(A, Y).   } $$
If $j$ is a cofibration (in the sense of \defref{def: digital cofibration}, not in a ``based" sense),  then there are subdivisions $S( -, k)$ and $S(I_m, l)$, and  a filler $\overline{H}\colon S(Z, k) \times S(I_m, l) \to \map_*( S(X, k), Y)$ in the following commutative diagram:
$$
\xymatrix{
S(Z, k) \times \{0\} \ar[d]_{\mathrm{id}\times i} \ar[r]^-{\rho_k\times \mathrm{id}} & Z \times \{0\} \ar[r]^-{f} \ar[d]_-{\mathrm{id}\times i}  &  \map_*(X, Y) \ar[d]^-{j^*} \ar[r]^-{ (\rho_k)^* } &  \map_*( S(X, k), Y) \ar[d]^{(S(j))^*}\\
S(Z, k) \times S(I_m, l)\ar[r]_-{\rho_k \times \rho_l}  \ar@{.>}[rrru]^-{\overline{H}} &Z \times I_m \ar[r]_-{H}   &  \map_*(A, Y)  \ar[r]_-{ (\rho_k)^*  } & \map_*( S(A, k), Y)   }
$$
\end{theorem}

\begin{proof}
Via the inclusions
$$\xymatrix{ \map_*(X, Y) \ar[d]^-{j^*} \ar[r]^-{\text{incl.}} & \map(X, Y) \ar[d]^-{j^*} \\
 \map_*(A, Y)  \ar[r]_-{\text{incl.}} & \map(X, Y)}$$
 (which is the ``restriction of a fibration" hinted at above), the given data yield  a commutative diagram
$$\xymatrix{
Z \times \{0\} \ar[r]^-{f} \ar[d]_-{\mathrm{id}_Z \times i}  & \map(X, Y) \ar[d]^-{j^*}\\
Z \times I_m \ar[r]_-{H}  & \map(A, Y),   } $$
and thus a filler $\overline{H}\colon S( Z, k)  \times S(I_m, l) \to  \map(S( X, k), Y)$ as in  Diagram (\ref{eq: fibration diag}) in \thmref{thm: Digital Borsuk}.  We simply observe that, under our hypotheses, the mapping spaces in the right-hand part of this diagram may be replaced by their based counterparts.  First, since $f$ and $H$ are assumed to have image in the based mapping spaces, we may replace $\map(X, Y)$, respectively $\map(A, Y)$, by $\map_*(X, Y)$, respectively $\map_*(A, Y)$, in the diagram.  Next, since $\rho_k\colon S(A, k) \to A$ is a based  map, we may replace the lower right  $\map(S(A, k), Y)$ by  $\map_*(S(A, k), Y)$.  Finally, since the diagram commutes, the image of $\overline{H}\circ (S(j))^*$ is contained in   $\map_*(S(A, k), Y)$, and it follows that the image of $\overline{H}$ is contained in $\map_*(S(X, k), Y)$.  For suppose we have $g \in \map_*(S(X, k), Y)$ that satisfies
$(S(j))^*(g) \in \map_*(S(A, k), Y)$.  Then $(S(j))^*(g) (ka_0) = y_0$, but we have  $(S(j))^*(g) (ka_0) = g\big( S(j)(ka_0) \big) = g(ka_0)$: the map $g$ must be a based map.  It follows that we have the commutative diagram asserted.
\end{proof}

\begin{definition}\label{def: based path fibration}
For any interval $I_M$, choose $\{0\} \in I_M$ as the basepoint.
For a based digital image $Y$, with basepoint $y_0 \in Y$,  let $\mathcal{P}_N Y$ denote the \emph{based} path space (of paths of length $N$), so that
$$\mathcal{P}_N Y = \{ \gamma \in P_N Y \mid \gamma(0) = y_0\}.$$
Also, let $\text{ev}_N\colon \mathcal{P}_N Y \to Y$ denote  the evaluation map $\text{ev}_N(\gamma) = \gamma(N)$.
\end{definition}

\begin{remark}\label{rem:based path fibration induced}
Just as for the other path fibrations considered already, the based path fibration may be identified with a map induced on \emph{based} function spaces.  Namely, if we take $I_N$ with $N \geq 2$, and consider the cofibration $j\colon \{0, N\} \to I_N$ as a based map, with $0$ as the basepoint in both $\{0, N\}$ and $I_N$, then we may identify  $\text{ev}_N\colon \mathcal{P}_NY \to Y$ and $j^* \colon \map_*(I_N, Y) \to \map_*(\{0, N\}, Y)$.  Notice that here, as in \remref{rem:TC fibration induced}, we want $N \geq 2$ so that no adjacency requirement constrains the value of $f(N)$, for a based map $f \in \map_*(\{0, N\}, Y)$.
\end{remark}

\begin{corollary}\label{cor: based path fibration}
For any based digital space $Y$, suppose given a commutative diagram
$$\xymatrix{
Z \times \{0\} \ar[r]^-{f} \ar[d]_-{\mathrm{id} \times i}  & \mathcal{P}_N Y\ar[d]^-{\text{ev}_N}\\
Z \times I_M \ar[r]_-{H}  & Y,    } $$
with $N \geq 2$.
Then there are subdivisions $S( Z, k)$, $S(I_N, k) = I_{kN+k-1}$, and $S(I_M, l)= I_{lM+l-1}$, and  a filler $\overline{H}\colon S(Z, k) \times I_{lM+l-1} \to \mathcal{P}_{kN+k-1}Y$ in the following commutative diagram:
$$
\xymatrix{
S(Z, k) \times \{0\} \ar[d]_{\mathrm{id}\times i} \ar[r]^-{\rho_k\times \mathrm{id}} & Z \times \{0\} \ar[r]^-{f} \ar[d]_-{\mathrm{id}\times i}  &  \mathcal{P}_N Y \ar[d]^-{\text{ev}_N} \ar[r]^-{ (\rho_k)^* } & \mathcal{P}_{kN+k-1}Y \ar[d]^{\text{ev}_{kN+k-1}}\\
S(Z, k) \times I_{lM+l-1}\ar[r]_-{\rho_k \times \rho_l}  \ar@{.>}[rrru]^-{\overline{H}} &Z \times I_M \ar[r]_-{H}   &  Y  \ar@{=}[r]  & Y   }
$$
\end{corollary}

\begin{proof}
Here we argue as in \corref{cor: TC fibration}, using the same cofibration $j\colon \{0, N\} \to I_N$ we used there.  Take basepoints and identify  $\text{ev}_N\colon \mathcal{P}_NY \to Y$ and $j^* \colon \map_*(I_N, Y) \to \map_*(\{0, N\}, Y)$ as in \remref{rem:based path fibration induced} above. Combine \thmref{thm: endpoint cofibration}  and \thmref{thm: Based Digital Borsuk} to obtain a commutative diagram as follows:
$$
\xymatrix{
S(Z, k) \times \{0\} \ar[d]_{\mathrm{id}\times i} \ar[r]^-{\rho_k\times \mathrm{id}} & Z \times \{0\} \ar[r]^-{f} \ar[d]_-{\mathrm{id}\times i}  &  \mathcal{P}_N Y \ar[d]^-{\text{ev}_N} \ar[r]^-{ (\rho_k)^* } & \mathcal{P}_{kN+k-1}Y \ar[d]^{(S(j))^*}\\
S(Z, k) \times I_{lM+l-1}\ar[r]_-{\rho_k \times \rho_l}  \ar@{.>}[rrru]^-{\overline{H}} &Z \times I_M \ar[r]_-{H}   &  Y  \ar[r]_-{ (\rho_k)^*  } &  \map_*( S(\{0, N\}, k), Y)    }
$$
Now use $i \colon  \{0, N\} \to S( \{0, N\}, k)$  just as in \corref{cor: TC fibration}.    Composing with $i^*\colon \map_*( S(\{0, N\}, k), Y) \to \map_*(\{0, N\}, Y) = Y$ on the bottom right-hand corner gives the result.
\end{proof}

\begin{remark}
Just as in \remref{rem: not surjective}, the evaluation map $\text{ev}_N\colon \mathcal{P}_NY \to Y$ will generally not be surjective: there will be points in $Y$ far enough away from the basepoint $y_0$ so that they cannot be reached by a path of length $N$.  However, for sufficiently large $k$, the map $(S(j))^* \colon \map_*( S(I_N, k), Y) \to \map_*( S(\{0, N\}, k), Y)
$ will be surjective as long as $Y$ is connected.
\end{remark}

\section{Covering Paths and Homotopies in the Diamond}\label{sec: Diamond Winding Number}

In this section we present  some results  that focus specifically on paths and loops in the Diamond  (see \propref{prop: diamond non-contractible}, which we will generalize below).  Although these results do not follow from our general results on cofibrations, they seem appropriate to include here as they deal with notions such as covering, path and homotopy lifting, and the winding number, in the digital setting.   Also, we will apply these results to calculate a certain invariant of the Diamond in the next section.  We make a further application of the results of this section in \cite{LOS19c}.

Covering spaces have appeared in the digital topology literature; our results here are similar in approach to results of 
\cite{Ha05, B-K10}, for instance.  Our results here do not follow from previous work, however, for the usual reasons: the general results of \cite{Ha05, B-K10} involve various choices of adjacency, whereas we use a fixed adjacency; the notion of homotopy we use here differs from that used in  \cite{B-K10}, for example.

Recall that the Diamond $D \subseteq \Z^2$ consists of the four points $\{(\pm 1, 0), (0, \pm1) \}$ with adjacencies determined as a digital image in $\Z^2$.  We think of the Diamond as  the prototypical digital circle.  We may equally well represent the points of the Diamond as complex numbers $\{ \pm1, \pm i\}$, and hence as $\{e^{k\pi i/2} \mid k = 0, 1, 2, 3\}$.  Then we have an adjacency preserving projection
$$p \colon \Z \to D,$$
defined by $p(n) =  e^{n\pi i/2}$. This restricts to give a map of digital images $p\colon [a, b] \to D$ for any interval $[a, b] \subseteq \Z$.  This projection is a digital version of the standard covering projection $\R \to S^1 \subseteq \C$ given by  $x \mapsto e^{2\pi x i}$  As we will see, the digital version shares some of the properties of its topological counterpart.

\begin{lemma}[path-lifting property]\label{lem: Diamond path lifting}
Suppose we are given a commutative diagram
$$\xymatrix{
\{0\} \ar[r]^-{f} \ar[d]_-{i}  & [-M, M] \ar[d]^-{p}\\
I_N \ar[r]_-{\alpha}  \ar@{.>}[ru]^-{\overline{\alpha}}  & D,    } $$
in which $p$ is the projection as above, and $M$ is sufficiently large such that we have $[f(0) -N, f(0)+N] \subseteq [-M, M]$.  Then there is a unique filler $\overline{\alpha} \colon I_N \to [-M, M]$, namely, a path of length $N$ in $[-M, M]$ that lifts $\alpha$ through $p$ and starts at $f(0) \in p^{-1}\big(\alpha(0)\big)$.
\end{lemma}

\begin{proof}
It is intuitively clear that there is a lift if one pictures $\Z$ as embedded  as a ``helix" in $\Z^3$, with $n \mapsto\big( \text{re}( e^{n\pi i/2}), \text{im}( e^{n\pi i/2}), n\big)$, and $p$ just projection onto the first two coordinates (as the covering  $\R \to S^1$ is usually pictured).  The condition on $M$ is simply so that we have enough room to accomodate this obvious lift.  We will progressively construct this lift, and show it is unique at the same time.

We work by induction over the length of a lift.  The initial point of any lift is specified: $\overline{\alpha}(0) = f(0)$.  So, for $k$ with $0 \leq k \leq N-1$, inductively suppose we have defined $\overline{\alpha}(s)$ for $s = 0, \ldots, k$, so that $p\circ  \overline{\alpha}(s) = \alpha(s)$ for $s = 0, \ldots, k$ and furthermore, if  $\alpha' \colon I_k \to \Z$ is any other path of length $k$ that starts at $f(0)$ and lifts $\alpha|_{I_k}$, then $\alpha'(s) = \overline{\alpha}(s)$ for  $s = 0, \ldots, k$.    Suppose that $\alpha(k) =  e^{r\pi i/2}$ for $r \in \{ 0, 1, 2, 3\}$. Then $\overline{\alpha}(k) = 4q+r$ for some $q \in \Z$. Notice that continuity of $\overline{\alpha} |_{I_k}\colon I_k \to \Z$ implies that we have $-M < f(0)-k\leq 4q+r \leq f(0) + k<M$, so we still have room to extend $\overline{\alpha}$ in either direction.  Since $\alpha(k+1) \sim_D \alpha(k)$, we have $\alpha(k+1) =   e^{(r+\epsilon)\pi i/2}$, where $\epsilon \in \{\pm 1, 0\}$.  So extend $\overline{\alpha}$ to $\overline{\alpha}(k+1) = 4q+r + \epsilon$.  Then
$\overline{\alpha}(k+1) \sim_{[-M, M]} \overline{\alpha}(k)$, so $\overline{\alpha} |_{I_{k+1}}\colon I_{k+1} \to \Z$ is continuous and lifts $\alpha|_{I_{k+1}}$.  Furthermore,   if  $\alpha' \colon I_{k+1 }\to \Z$ is any other path of length $k+1$ that starts at $f(0)$ and lifts $\alpha|_{I_{k+1}}$, then we have by our inductive assumption of uniqueness that $\alpha'(k) = \overline{\alpha}(k)$, and thus we have $\alpha'(k+1) \sim_{[-M, M]} \alpha'(k) =  \overline{\alpha}(k) = 4q+r$.  Since $p\circ \alpha'(k+1) = \alpha(k+1)  = e^{(r+\epsilon)\pi i/2}$, we have $\alpha'(k+1) = 4Q + r + \epsilon$, for some $Q$.    But if $4Q + r + \epsilon \sim_{[-M, M]} 4q + r$, for some $\epsilon  \in \{\pm 1, 0\}$, we must have $Q = q$, and so $\alpha'(k+1) = \overline{\alpha}(k+1)$.  This completes the inductive step.  the result follows, by induction.
\end{proof}

\begin{corollary}[winding number for loops in $D$]\label{cor: winding number}
For each loop in $D$, that is, for each path $\alpha\colon I_N \to D$ with $\alpha(0) = \alpha(N)$, there is a well-defined integer $w(\alpha)$ (the \emph{winding number} of the path $\alpha$---in fact a multiple of $4$), given by
$w(\alpha) = \overline{\alpha}(N) - \overline{\alpha}(0)$, where $\overline{\alpha}$ is the unique lift of $\alpha$ guaranteed by \lemref{lem: Diamond path lifting} for any choice
of initial point $\overline{\alpha}(0)$.
\end{corollary}

\begin{proof}
We need only check that, for different choices of initial point, the number stays the same.  So suppose that $\overline{\alpha}$ is the lift of $\alpha$ that starts at $\overline{\alpha}(0) = n_0$.  Suppose that $\alpha'$ is the lift of $\alpha$ that starts at $\alpha'(0) = m_0$.  Since $p(n_0) = p(m_0)$, these two initial points must differ by some multiple of $4$.  Then the path (in some suitably large interval) defined by $s \mapsto   \overline{\alpha}(s) + m_0 - n_0$ lifts $\alpha$ through $p$, and starts at $m_0$.  Therefore, by uniqueness, we must have
$\alpha'(s) = \overline{\alpha}(s) + m_0 - n_0$, and hence $\overline{\alpha}(N) - \overline{\alpha}(0) = \alpha'(N) - \alpha'(0)$: the value of $w(\alpha)$ is well-defined.
\end{proof}

\begin{lemma}[homotopy-lifting property]\label{lem: Diamond homotopy lifting}
Suppose given a commutative diagram
$$\xymatrix{
I_N \times \{0\} \ar[r]^-{\overline{\alpha}} \ar[d]_-{i}  & [-L, L] \ar[d]^-{p}\\
I_N  \times I_M \ar[r]_-{H}  \ar@{.>}[ru]^-{\overline{H}}  & D,    } $$
in which $p$ is the projection as above,       $H \colon I_N \times I_M \to D$ is a homotopy of paths of length $N$ in $D$ that starts at $H(s, 0) = \alpha(s)$, $\overline{\alpha} \colon I_N \to [-L, L]$ is the unique lift of $\alpha$ through $p$ for a given initial point $\overline{\alpha}(0)$,  and $L$ is sufficiently large such that we have $[\overline{\alpha}(0) -(N+M), \overline{\alpha}(0)+N+M] \subseteq [-L, L]$.  Then there is a unique filler $\overline{H} \colon I_N \times I_M \to [-L, L]$, namely, a homotopy  of length $M$ in $[-L, L]$ that lifts $H$ through $p$ and starts at $\overline{\alpha}$.
\end{lemma}

\begin{proof}
We follow a similar strategy to that of the proof of \lemref{lem: Diamond path lifting}, and show existence and uniqueness of the lift together.  We construct the (unique) lift $\overline{H} \colon I_N \times I_M \to [-L,L]$ horizontal row-wise ($M$ rows, each of length $N$), working inductively.  First, any lift is specified on the bottom row: we have $\overline{H}(s, 0) = \overline{\alpha}(s)$ for $s = 0, \ldots, N$.  If $H'\colon I_N \times I_M \to  [-L, L]$ is any other lift of $H$, then we must have   $H'(s, 0) = \overline{H}(s, 0) = \overline{\alpha}(s)$ for $s = 0, \ldots, N$.  This starts the induction.  Now assume inductively that we have constructed $\overline{H} \colon I_N \times I_l \to [-L, L]$ for some $l$ with $0 \leq l \leq M-1$, such that
$p\circ \overline{H} = H\colon I_N \times I_l \to D$, $\overline{H}(s, 0) = \overline{\alpha}(s)$ for $s = 0, \ldots, N$, and, furthermore, that if $H'$ is any other lift of $H$ on $I_N \times I_l$ that also satisfies $H'(s, 0) = \overline{\alpha}(s)$ for $s = 0, \ldots, N$, then we must have  $H'(s, t) = \overline{H}(s, t)$ for $(s, t) \in I_N \times I_l$.   The inductive step consists of extending the definition of a suitable $\overline{H}$ on the row $(s, l+1)$ for $s = 0, \ldots, N$.

Start by defining $\overline{H}(0, l+1)$.  Here, we have no choice, because $\overline{H}(0, t)$ is defined for $t = 0, \ldots l$, and is the unique lift of  the path $H(0, t)$ for  $t = 0, \ldots l$.  Thus $\overline{H}(0, l+1)$ must be defined so as to extend this lift to the unique lift of $H(0, t)$ for  $t = 0, \ldots l+1$.  This determines a value for $\overline{H}(0, l+1)$ that satisfies  $\overline{H}(0, l+1) \sim_{[-L, L]} \overline{H}(0, l)$.  Furthermore, notice that any other $H'$ that lifts $H$  and agrees with $\overline{H}$ on $I_N \times I_l$ (or even just $\{0\} \times I_l$) must satisfy $H'(0, l+1) = \overline{H}(0, l+1)$, since both lift the same path and start at the same initial point.  We claim that, not only do we have
$\overline{H}(0, l+1) \sim_{[-L, L]} \overline{H}(0, l)$, but also the adjacency $\overline{H}(0, l+1) \sim_{[-L, L]} \overline{H}(1, l)$.  For suppose that $H(0, l+1) = e^{ri\pi/2}$, for $r \in \{0, 1, 2, 3\}$.  Then $\overline{H}(0, l+1) = 4n + r$ for some $n$.    Since $(0, l+1)$, $(0, l)$, and $(1, l)$ are pair-wise adjacent in $I_N \times I_M$, we have
$$H(0, l) =  e^{\big(r+\epsilon(0, -1)\big)i\pi/2} \quad \text{and} \quad H(1, l) =  e^{\big(r+\epsilon(1, -1)\big)i\pi/2},$$
for  $\epsilon(0, -1), \epsilon(1, -1) \in \{ \pm1, 0\}$ that satisfy inequalities
$$|\epsilon(0, -1)| \leq 1, \quad |\epsilon(1, -1)| \leq 1, \quad \text{and} \quad |\epsilon(0, -1) - \epsilon(1, -1)| \leq 1.$$
It follows that
$$\overline{H}(0, l) =  4m + r+\epsilon(0, -1) \quad \text{and} \quad \overline{H}(1, l) =  4p+ r+\epsilon(1, -1),$$
for some $m$ and $p$, but the adjacency $\overline{H}(0, l+1) \sim_{[-L, L]} \overline{H}(0, l)$ implies that $n = m$, and the adjacency $\overline{H}(0, l) \sim_{[-L, L]} \overline{H}(0, l)$
 implies that $m = p$.  So we have
$$\overline{H}(0, l+1) =  4n + r, \quad  \overline{H}(0, l) =  4n + r+\epsilon(0, -1) \quad \text{and} \quad \overline{H}(1, l) =  4p+ r+\epsilon(1, -1),$$
with  the $\epsilon(i, j)$ satisfying  the relations above, which means that $\overline{H}(0, l+1)$, $\overline{H}(0, l)$, and $\overline{H}(1, l)$ are pair-wise adjacent in $[-L, L]$.

So far, we have extended the definition of $\overline{H}$ to $I_N \times I_l \cup \{ (0, l+1)\}$, and shown the extension is both unique and continuous on $I_N \times I_l \cup \{ (0, l+1)\}$.  Now we proceed in the same way along the row, using a secondary induction.  Suppose inductively that we have extended $\overline{H}$ to $I_N \times I_l \cup \{ (0, l+1), \dots, (u, l+1)\}$, for some $u$ with $0 \leq u \leq N-1$, in such a way that it is a continuous lift of $H$ on  $I_N \times I_l \cup \{ (0, l+1), \dots, (u, l+1)\}$ and furthermore, that if $H'$  is another lift of $H$ on $I_N \times I_l \cup \{ (0, l+1), \dots, (u, l+1)\}$ that agrees with $\overline{H}$ on $I_N \times I_l$, then $H' = \overline{H}$ on $I_N \times I_l \cup \{ (0, l+1), \dots, (u, l+1)\}$.  This secondary induction starts with $u = 0$, which is what we just showed.  For the inductive step, we define $\overline{H}$ on $(u+1, l+1)$, and show it is continuous and unique.  For the definition, we have no choice. For $\overline{H}(s, l+1)$ is defined for $s = 0, \ldots u$, and is the unique lift of  the path $H(s, l+1)$ for  $s = 0, \ldots u$.  Thus $\overline{H}(u+1, l+1)$ must be defined so as to extend this lift to the unique lift of $H(s, l+1)$ for  $s = 0, \ldots u+1$.  This determines a value for $\overline{H}(u+1, l+1)$ that satisfies  $\overline{H}(u+1, l+1) \sim_{[-L, L]} \overline{H}(u, l+1)$.  Furthermore, notice that any other $H'$ that lifts $H$  and agrees with $\overline{H}$ on $I_N \times I_l$ first, must satisfy $H'(0, l+1) = \overline{H}(0, l+1)$ by the first part of this argument, and hence second, must satisfy $H'(s, l+1) = \overline{H}(s, l+1)$, for $s = 0, \ldots, u+1$ since both lift the same path and start at the same initial point.  That is, $\overline{H}$ is the unique lift on $I_N \times I_l \cup \{ (0, l+1), \dots, (u+1, l+1)\}$.  To complete the inductive step, we must check that $\overline{H}$ is continuous on $I_N \times I_l \cup \{ (0, l+1), \dots, (u+1, l+1)\}$.  For this, suppose  that we have  $H(u, l+1) = e^{ri\pi/2}$, for some $r \in \{0, 1, 2, 3\}$.  Then $\overline{H}(u+1, l+1) = 4n + r$ for some $n$.    Now on those points $(s, t)$ in $I_N \times I_l \cup \{ (0, l+1), \dots, (u+1, l+1)\}$ adjacent to $(u+1, l+1)$, the continuity of $H$ means that we may display the values of $H$ in Table~\ref{tab:values of H} for  $\epsilon(i, j) \in \{ \pm1, 0\}$ that satisfy inequalities
$$|\epsilon(i, j)| \leq 1,  \quad \text{and} \quad |\epsilon(i, j) - \epsilon(i', j')| \leq 1,$$
whenever $(i, j) \sim_{\Z^2} (i', j')$.
\begin{table}
\begin{tabular}{c|c|c|c}
 &  $s = u$ & $s = u+1$ & $s = u+2$\\
  & & & \\
 \hline
  & & & \\
   $t = l+1$ & $e^{\big(r+\epsilon(-1, 0)\big)i\pi/2}$ & $H(u+1, l+1) = e^{ri\pi/2}$ & \\
  & & & \\
   \hline
     & & & \\
 $t = l$ & $e^{\big(r+\epsilon(-1, -1)\big)i\pi/2}$ & $e^{\big(r+\epsilon(0, -1)\big)i\pi/2}$ & $e^{\big(r+\epsilon(1, -1)\big)i\pi/2}$\\
   & & & \\
\end{tabular}
\caption{\label{tab:values of H}Values of $H(s, t)$ for $(s, t)$ adjacent to $(u+1, l+1)$ and in $I_N \times I_l \cup \{ (0, l+1), \dots, (u+1, l+1)\}$.}
\end{table}
Correspondingly, since $\overline{H}$ lifts $H$, we have values of $\overline{H}$ on the same points in $I_N \times I_l \cup \{ (0, l+1), \dots, (u+1, l+1)\}$ displayed in Table~\ref{tab:values of barH} for $n_{(i, j)} \in \Z$, and the same $\epsilon(i, j)$ that satisfy the identities above.

\begin{table}
\begin{tabular}{c|c|c|c}
 &  $s = u$ & $s = u+1$ & $s = u+2$\\
  & & & \\
 \hline
  & & & \\
 $t = l+1$ & $ 4 n_{(-1, 0)} +r +\epsilon(0, -1)$ & $\overline{H}(u+1, l+1) = 4n+r$ & \\
  & & & \\
 \hline
  & & & \\
 $t = l$ & $4 n_{(-1, -1)}+ r +\epsilon(-1, -1)$ & $4 n_{(0, -1)} + r +\epsilon(0, -1)$ & $4 n_{(1, -1)} + r + \epsilon(1, -1)$\\
  & & & \\
\end{tabular}
\caption{\label{tab:values of barH}Values of $\overline{H}(s, t)$ for $(s, t)$ adjacent to $(u+1, l+1)$ and in $I_N \times I_l \cup \{ (0, l+1), \dots, (u+1, l+1)\}$.}
\end{table}

Now we already said, above, that $\overline{H}(u+1, l+1) \sim_{[-L, L]} \overline{H}(u, l+1)$, from the way in which we defined $\overline{H}(u+1, l+1)$.  Therefore, we must have
 $n_{(-1, 0)} = n$.  Furthermore, $\overline{H}$ is continuous on $I_N \times I_l \cup \{ (0, l+1), \dots, (u, l+1)\}$, and hence we must have
 $n_{(-1, 0)} = n_{(-1, -1)} = n_{(0, -1)} =n_{(1, -1)}$, and so all of the $n(i, j) = n$.  Since the relations obeyed by the $\epsilon(i, j)$ are those for adjacency, it follows that
 $\overline{H}(u+1, l+1)$ is adjacent in $[-L, L]$ to each of the values $\overline{H}(s, t)$ displayed, and thus the extension of $\overline{H}$ to  $I_N \times I_l \cup \{ (0, l+1), \dots, (u+1, l+1)\}$ is continuous.  This completes the secondary inductive step, and so by the secondary induction, we have that $\overline{H}$ extends uniquely to a continuous lift of $H$ on $I_N \times I_{l+1}$.  In turn, this completes  the (primary) inductive step, and it follows by induction that $\overline{H}$ extends uniquely to a lift of $H$ on $I_N \times I_M$.
\end{proof}

\begin{proposition}\label{prop: loops in Diamond}
If two loops of length $N$ in the Diamond are homotopic, then they have the same winding number.  In particular, no loop in the Diamond with non-zero winding number is homotopic to a constant loop.
\end{proposition}

\begin{proof}
Suppose that $\alpha, \beta \colon I_N \to D$ are homotopic by an $M$-stage homotopy.  Then the $M$-stage homotopy $H\colon I_N \times I_M \to D$ may be regarded as a succession of $1$-stage homotopies (or, adjacencies in the path space $P_ND$), thus: $\alpha \approx_1 \alpha_1 \approx \cdots \approx_1 \alpha_M = \beta$, with $\alpha_i(s) = H(s, i)$, for $i = 1, \ldots M$.    So, without loss of generality, suppose that $\alpha$ and $\beta$ are $1$-homotopic (adjacent in the path space $P_ND$).   Then the $1$-homotopy
$H\colon  I_N \times I_1 \to D$ from $\alpha$ to $\beta$ lifts, as in \lemref{lem: Diamond homotopy lifting}, to a $1$-homotopy $\overline{H} \colon I_N \times I_1 \to [-L, L]$ from $\overline{\alpha}$ to $\overline{\beta}$.  Then we have $\overline{\alpha}(0) = \overline{H}(0, 0) \sim_{[-L, L]} \overline{H}(0, 1) = \overline{\beta}(0)$, and $\overline{\alpha}(N) = \overline{H}(N, 0) \sim_{[-L, L]} \overline{H}(N, 1) = \overline{\beta}(N)$.  This means that we have $\overline{\beta}(0) = \overline{\alpha}(0) + \epsilon$ and $\overline{\beta}(N) = \overline{\alpha}(N) + \epsilon'$, with $\epsilon, \epsilon' \in \{\pm1, 0\}$.  Then $w(\beta) = w(\alpha) + (\epsilon' - \epsilon)$ but, since the winding number of a loop must be a multiple of $4$, and $|(\epsilon' - \epsilon)| \leq 2$, we must have $w(\beta) = w(\alpha)$.  In other words, adjacent loops in $P_ND$ have the same winding number.  Hence, homotopic loops in  $P_ND$ have the same winding number.

A constant loop in $D$ is lifted through $p$ to a constant path, with winding number zero.  The last assertion follows.
 \end{proof}

\section{Digital Category}\label{sec: cat and TC}

We indicate how other homotopy-theoretic notions may be  developed in the digital setting.  Here, we build on the ideas so far, to give a preliminary treatment of  Lusternik-Schnirelmann category.  The sequence of definitions and results presented here follows, \emph{mutatis mutandis}, a typical presentation of these ideas in the topological setting.   On the digital side, all the main notions introduced in this section incorporate subdivision in a basic way.  This is consistent with our philosophy that, if we are to have an interesting homotopy theory in the digital setting, then we need to use subdivision when adapting notions from the topological setting into the digital.
In fact, the notion of Lusternik-Schnirelmann category has appeared in the digital literature previously (see \cite{Bo-Ve18}).  But the approach of \cite{Bo-Ve18} is to translate the topological notion directly into the digital setting, and the common drawbacks of such an approach are apparent there:  it is hard to make use of invariance under homotopy equivalence, since digital images are rarely homotopy equivalent; the rigidity of the invariant is such that general results are hard to obtain.

We first give a less rigid version of contractibility than that of \defref{def:contractible}.

\begin{definition}\label{def: sub contract}
We say that $X$ is \emph{subdivision-contractible} if, for some subdivision $S(X, k)$ of $X$, and some $x_0 \in X$, and some $N$, we have a homotopy $H\colon  S(X, k) \times I_N \to  X$  with $H(x, 0) = \rho_k(x)$, and $H(x, N) = x_0$.
\end{definition}

In the ordinary, topological setting, Lusternik-Schnirelmann category is a numerical homotopy invariant  that plays a prominent role in many questions concerning dynamics and smooth functions on manifolds and is a well-known topic in homotopy theory (see \cite{CLOT}).  For a topological space $X$, it is a natural number denoted by $\cat(X)$ that may be defined  as one less than the minimum number of sets in a covering of $X$ by open sets, each of which is contractible in $X$. Thus it may be viewed as an index of how complicated $X$ is, since it corresponds to the smallest number of ``simple pieces" that $X$ may be assembled from.

We will adapt this covering definition to give a digital version of  Lusternik-Schnirelmann category and show that it is a numerical homotopy invariant at least for 2D digital images (see \thmref{thm: 2D dcat well-def}). This fact allows us to tell digital images apart up to homotopy in the sense that if $\cat(X) \not= \cat(Y)$,  then $X$ and $Y$ cannot be homotopy equivalent. Furthermore, the notion provides insight into how a digital image may be decomposed into simpler pieces, in a way that could be useful for various kinds of construction.

\begin{definition}\label{def: U contractible in X}
Let $i \colon U \to X$ be an inclusion of digital images.

We say that $U$ is \emph{categorical in $X$} if, for some $x_0 \in X$, and some $N$,  we have a homotopy
$H\colon  U \times I_N \to  X$  with $H(u, 0) = x_0$, and $H(u, N) = i(u)$, for all $u \in U$.

We say that $U$ is \emph{subdivision-categorical in $X$} if, for some subdivision $S(U, k)$ of $U$, some $x_0 \in X$, and some $N$,  we have a homotopy
$H\colon  S(U, k) \times I_N \to  X$  with $H(u', 0) = x_0$, and $H(u', N) = i\circ \rho_k(u')$, for all $u' \in S(U, k)$.
\end{definition}

\begin{example}\label{ex: punctured sphere}
Imagine a (topological) sphere with the north pole deleted.  This is (topologically) contractible.  However, a contracting homotopy---contracting everything to the south pole, say---would need to enlarge ``parallels'' of latitude from the northern hemisphere to pass over the equator, and then shrink them to a point.  Digitally, a homotopy cannot enlarge a circle, but if we allow for a subdivision, then we may enlarge.  So a suitable digital analogue of this situation provides an example of a digital image $X$ (a digital sphere with north pole removed) that is subdivision-contractible but not contractible.  Furthermore, in the same example, the tropic of Cancer (``parallel'' of latitude at approx.~$23.5^\circ$ north), say,  is (topologically) contractible in the deleted sphere.  However, a homotopy that contracts---contracting the tropic of Cancer to the south pole in the deleted sphere, say---would need to enlarge, and then shrink it to a point.   So a suitable digital analogue of this situation also provides an example of a subset  $U$ (a digital tropic of Cancer) that is subdivision-categorical, yet not categorical, in $X$.
\end{example}

\begin{definition}[Digital Category]\label{def: d-cat}
The \emph{digital category of $X$}, denoted by $\dcat(X)$,  is the smallest $n\geq 0$ for which there is a covering of $X$ by $n+1$ subsets that are subdivision-categorical in $X$.  Note that, since we consider only finite digital images, we always have a (finite) value for $\dcat(X)$.
\end{definition}

\begin{example}
If $X$ is contractible, then $X$ is subdivision-contractible and we have $\dcat(X)= 0$.
If $\dcat(X)= 0$, then $X$ is subdivision-contractible, tautologically from the definition (but generally not contractible---see \exref{ex: punctured sphere}).
\end{example}

\begin{proposition}\label{prop: cat of D}
The Diamond $D$ has $\dcat(D) = 1$.
\end{proposition}

\begin{proof}
The diamond $D$ can  be covered by two subdivision-contractible sets $\{D-\{p\}, \{p\}\}$, or $\{D-\{p\}, D-\{\overline{p}\}\}$ for any point $p\in D$ and $\overline{p}$ the ``antipode" of $p$.  Indeed, in either case, the two sets are actually contractible.  This gives $\dcat(D) \leq 1$.  On the other hand, we claim that the Diamond is not subdivision-contractible.  For suppose we had some subdivision $S(D, k)$ and a contracting homotopy $H\colon  S(D, k) \times I_N \to  D$  with $H(x, 0) = \rho_k(x)$, and $H(x, N) = x_0$.  Then the ``innermost" points in the subdivision give a loop in $D$.  Specifically, the original four points of $D$ will correspond to $(\pm k, 0), (0, \pm k)$, and the four segments
$\{ (k, t) | t = 0, \dots, k-1\}  \cup \{ (k-1-s, k) | s = 0, \dots, k-1\} \cup \{ (-1, k-1-t) | t = 0, \dots, k-1\} \cup \{ (s, -1) | s = 0, \dots, k-1\} \subseteq S(D, k)$ give a loop of length $4k$, $\alpha\colon I_{4k} \to D$, in $S(D, k)$.  But composing with $\rho_k\colon S(D, k) \to D$ gives a loop in $D$, and the composition
$$H\circ(i\times\text{id})\colon I_{4k} \times I_N \to D \times I_N \to D$$
is a contracting homotopy that starts at the loop $\rho_k\circ\alpha$ and ends at a constant loop.  The loop $\rho_k\circ\alpha$ clearly has winding number, as we defined it in
\corref{cor: winding number}, of $1$.  Whereas a constant loop in $D$ has winding number of $0$.  These two loops in $D$ cannot be homotopic, by \propref{prop: loops in Diamond}.  This contradicts the assumption of a subdivision-contracting homotopy of $D$.  So $D$ is not subdivision-contractible and we have  $\dcat(D) \geq 1$.  Thus $\dcat(D) =1$.
\end{proof}

\begin{proposition}\label{prop: categorical = secategorical}
Let $i\colon U \to X$ be an inclusion of  digital images.  The following are equivalent.
\begin{enumerate}
\item $U$ is subdivision-categorical in $X$.
\item There is some subdivision $S(U, k)$ and some $N$, for which there exists a filler $\sigma\colon S(U, k) \to \mathcal{P}_{N} X$ in the following commutative diagram:
$$\xymatrix{  &  \mathcal{P}_{N} X \ar[d]^{\text{ev}_N} \\
S(U, k) \ar@{.>}[ru]^-{\sigma} \ar[r]_-{i\circ \rho_k} & X.}$$
\end{enumerate}
\end{proposition}

\begin{proof}
(1) $\implies$ (2): Suppose we have a homotopy $H \colon S(U, L) \times I_M \to X$ that contracts $S(U, L)$ to $x_0$ in $X$,  in the sense of  \defref{def: U contractible in X}.   Then we have a commutative diagram
$$\xymatrix{
S(U, L) \times \{0\} \ar[r]^-{f} \ar[d]_-{\mathrm{id} \times i}  & \mathcal{P}_n X\ar[d]^-{\text{ev}_n}\\
S(U, L) \times I_M \ar[r]_-{H}  & X,    } $$
in which $f(u', 0) = c_{x_0}$, the constant path at the point $x_0 \in X$, for each $u' \in S(U, L)$.
Notice that, WLOG, we may suppose  that $n \geq 2$ (indeed, we may take $n=2$ for our purposes). So from \corref{cor: based path fibration}, we have
subdivisions $S\big( S(U, L), K\big) = S(U, k)$ with $k = LK$, $S(I_n, K) = I_{N}$ with $N = Kn+K-1$, and $S(I_M, l)= I_{lM+l-1}$, and  a filler $\overline{H}\colon S(U, k) \times I_{lM+l-1} \to \mathcal{P}_{N}X$ in the following commutative diagram:
$$
\xymatrix{
S(U, k) \times \{0\} \ar[d]_{\mathrm{id}\times i} \ar[r]^-{\rho_k\times \mathrm{id}} & S(U, L) \times \{0\} \ar[r]^-{f} \ar[d]_-{\mathrm{id}\times i}  &  \mathcal{P}_n X \ar[d]^-{\text{ev}_n} \ar[r]^-{ (\rho_k)^* } & \mathcal{P}_{N} X\ar[d]^{\text{ev}_N}\\
S(U, k) \times I_{lM+l-1}\ar[r]_-{\rho_K \times \rho_l}  \ar@{.>}[rrru]_-{\overline{H}} &S(U, L) \times I_M \ar[r]_-{H}   &  X  \ar@{=}[r]  & X   }
$$
So define $\sigma\colon S(U, k) \to  \mathcal{P}_{N} X$ by $\sigma(u') = \overline{H}(u', lM+l-1)$ for $u' \in S(U, k)$, which clearly gives a continuous map.  Then we have $\text{ev}_N\circ \sigma(u') = H\circ (\rho_K \times \rho_l)(u', lM+l-1) = H\big( \rho_K(u'), M  \big) =  i\circ\rho_L\circ \rho_K(u') =  i\circ\rho_k(u') $, as required.

(2) $\implies$ (1): Given $\sigma\colon S(U, k) \to \map_*(I_N, X)$ as in the diagram, define a homotopy $H \colon S(U, k) \times I_N \to X$ by $H(u', t) = \sigma(u')(t)$.  Then $H$ is continuous by \propref{prop: exponential law}.  Furthermore, we check directly that
$H(u', 0) = \sigma(u')(0) = x_0$---since $\sigma$ maps to the based path space, and $H(u, N) = \text{ev}_N\circ\sigma(u') = i\circ\rho_k(u')$.
\end{proof}

\begin{corollary}\label{cor: d-cat = d-secat}
Let $X$ be a digital image.  Then $\dcat(X)$ equals the smallest $n\geq 0$ for which there is a covering of $X$ by $n+1$ subsets $U_0, \ldots, U_n$, for each of which there is some subdivision $S(U_i, k_i)$ and some $N_i$, for which there exists a filler $\sigma_i\colon S(U_i, k_i) \to \mathcal{P}_{N_i} X$ in the following commutative diagram:
$$\xymatrix{  &  \mathcal{P}_{N_i} X \ar[d]^{\text{ev}_{N_i} } \\
S(U_i, k_i) \ar@{.>}[ru]^-{\sigma_i} \ar[r]_-{i\circ \rho_k} & X.}$$
\end{corollary}

\begin{proof}
This is immediate from \propref{prop: categorical = secategorical} and \defref{def: d-cat}.
\end{proof}

\begin{remark}
The reader familiar with Lusternik-Schnirelmann category and surrounding topics will recognize in the above a nascent notion of \emph{sectional category} in the digital setting.  We avoid attempting a general definition of sectional category here, since we do not really need the general notion for our immediate purposes and furthermore, we do not yet have a general definition of fibration.
\end{remark}

Using a result from \cite{LOS19b}, we can establish that $\dcat$ is an invariant of homotopy type amongst two-dimensional (2D) digital images.

\begin{theorem}\label{thm: 2D dcat well-def}  Suppose $X \subseteq \Z^2$ is a 2D digital image and that we have maps $f\colon X \to Y$ and $g\colon Y \to X$ with a homotopy $H\colon X \times I_M\to X$ from $\text{id}_X$ to $g\circ f$.    Then $\dcat(X) \leq \dcat(Y)$.  If $X$ and $Y$ are both 2D and homotopy equivalent, then we have $\dcat(X) = \dcat(Y)$.
\end{theorem}

\begin{proof} The given homotopy $H\colon X \times I_M\to X$ satisfies $H(x,0)= x$ and $H(x,M)=g\circ f(x)$.  Suppose $U \subseteq Y$ is subdivision-categorical in $Y$ and write $C\colon S(U, k) \times I_N \to Y$ for the contracting  homotopy with $C(u',0)=\rho_k(u')$ and $C(u',N)=y_0$ for some $y_0 \in Y$.     Define $V = f^{-1}(U) \subseteq X$.  Since $X$ (and hence $V$) is 2D, Proposition 5.7 of \cite{LOS19b} gives  a map $\widehat{f}\colon S(V, k+1) \to S(U, k)$ with $f\circ \rho_{k+1} = \rho_k\circ \widehat{f}\colon S(V, k+1) \to U$.   Now define a homotopy $G \colon S(V, k+1) \times I_{M+N} \to X$  by
$$
G(v',t):=\begin{cases}
H(\rho_{k+1}(v'),t) & \text{for } 0\leq t\leq M\\
g\big(C(\widehat{f}(v'),t-M)\big) &  \text{for } M\leq t \leq M+N.
\end{cases}\\
$$
It is easy to check that $G(v',0)=\rho_{k+1}(v')$ and $G(v',M+N)=g(y_0)$.  Hence $V$ is subdivision-categorical in $X$.  If $\{U_0, \ldots, U_n \}$ is a subdivision-categorical cover of $Y$, then the same argument shows that  $\{V_0, \ldots, V_n \}$, with each $V_i = f^{-1}(U_i)$, is a subdivision-categorical cover of $X$.  It follows that we have $\dcat(X) \leq \dcat(Y)$.

If $X$ and $Y$ are both 2D and homotopy equivalent, then we may apply this argument to obtain $\dcat(X) \leq \dcat(Y)$, and also interchange the roles of $X$ and $Y$ to obtain $\dcat(Y) \leq \dcat(X)$.  Then we have $\dcat(X) = \dcat(Y)$.
\end{proof}

We suspect that $\dcat$ is an invariant of homotopy type without any restriction on the dimension of the digital images involved, but at present we are not able to establish this generally, beyond  the two dimensional case.   The argument given above would be sufficient  to do so, except that we would need an extension of the results of \cite{LOS19b} to higher-dimensional domains.  Furthermore, we believe  $\dcat$ is actually an invariant of a much weaker notion of equivalence than homotopy equivalence as we have defined it here.
For instance,  the two digital circles $D$ and $C$ of \exref{ex:basic digital images} are not homotopy equivalent, and yet we have  $\dcat(C) = \dcat(D) = 1$ (that $\dcat(C) =1$ may be seen by a direct argument like  that we gave for $\dcat(D) =1$). But to make any advance in this direction, again, would require basic results that extend those of \cite{LOS19b} to higher-dimensional domains. For these reasons, we do not attempt even a preliminary treatment of  topological complexity in the digital setting here.

\section{Questions, Problems, Future Work}\label{sec: future}

We finish by posing some questions and more general problems raised by our work here, and then discussing some aspects of future work that we anticipate as part of our digital homotopy theory project. 

\subsection{Questions} Our results about cofibrations in \secref{sec: cofibrations} suggest a number of questions.   We raised the following question after \lemref{lem: 1 x i cofibration}:

\begin{question}
Is the composition of cofibrations again a cofibration?
\end{question}

\noindent{}We have some very useful, but basic, examples of cofibrations in \thmref{thm: well-pointed interval} and \thmref{thm: endpoint cofibration}.  These may be leveraged to produce  other examples (see the comment after  \lemref{lem: 1 x i cofibration}); we suspect many other maps are cofibrations.  At this point, it seems reasonable to ask the following:   

\begin{question}
Is every inclusion $j \colon A \to X$ a cofibration?  Is every injection of digital images (not necessarily an inclusion) a cofibration?
\end{question}

\noindent{}It would be nice to remove the constraint on dimension from  \thmref{thm: 2D dcat well-def}.

\begin{question}
Is $\dcat$ an invariant of homotopy type?
\end{question}

\noindent{}Also, we may view subdivision of a digital image as essentially a process of ``enlarging."   Generally, our philosophy is that subdivision produces a digital image that should be viewed as equivalent to the original  (see the discussion below). 

\begin{question}\label{que: dcat subdiv}
Does $\dcat(X) = \dcat\big( S(X, k)\big)$?
\end{question}

\subsection{Problems}  \secref{sec: cofibrations} gives a satisfactory definition of cofibration but, as we mention at various points in the paper, \secref{sec: fibrations} establishes a modified homotopy lifting property for certain evaluation maps without actually giving a general definition of fibration.  Now \thmref{thm: Digital Borsuk} and \thmref{thm: Based Digital Borsuk} are quite general results that establish a modified homotopy lifting property for any map to which  they apply.  But not every map that we might imagine should be a ``fibration" will be of the form $j^*\colon  \map(X, Y) \to \map(A, Y)$ induced by some cofibration $j \colon A \to X$ (or $j^*\colon  \map_*(X, Y) \to \map_*(A, Y)$ in the based case).     

\begin{problem}
Formulate a definition of a digital fibration that incorporates the evaluation maps and their modified homotopy lifting property as in \secref{sec: fibrations} as special cases.
\end{problem}

\noindent{}Cofibrations and fibrations are two of the three distinguished types of map that go into an abstract, categorical notion of a homotopy theory, with  the third being a \emph{weak equivalence} (see, e.g.,  \cite{Hir03}).  
 
\begin{problem}
Is it possible to incorporate our notion of cofibration here into a suitable model category setting, and thereby place our emerging digital homotopy theory as a ``homotopy theory," in the technical (abstract) sense of a homotopy theory in a model category?
\end{problem}

\subsection{Discussion of Future Work} We indicate three directions for development within our larger digital homotopy theory project.

First, we anticipate developing a less rigid notion of homotopy equivalence.
For example, a circle might be represented as any of the digital images in \figref{fig:circles}.  In the figure we have indicated adjacencies in the style of a graph and included integer gridlines as dotted lines.
\begin{figure}[h!]
\centering
\includegraphics[trim=100 330 100 310,clip,width=\textwidth]{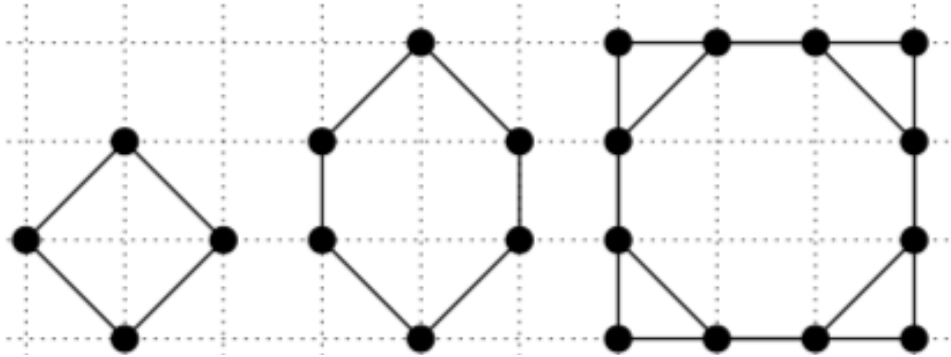}
\caption{\label{fig:circles} The diamond $D$ (left), another digital circle (middle), and a digital circle up to homotopy equivalence (right).}
\end{figure}
From a homotopy point of view, it seems reasonable to regard each of these as equivalent.  But the notion of homotopy equivalence that we have at present gives them as non-equivalent.  It gives, instead, a notion of equivalence comparable to that of isometry in a geometric setting, whereby circles of different sizes are not equivalent.  
Generally speaking, we seek to develop a notion of equivalence for digital images that is less rigid than homotopy equivalence and, instead, combines homotopy equivalence and subdivision.  For instance, we would like a notion of equivalence that treats a subdivision as equivalent to the original digital image (they are generally not homotopy equivalent).  Progress in this direction is represented by the results of \cite{LOS19c}.  In that paper, we introduce a notion of the fundamental group that features subdivision in a prominent role; we show that this fundamental group is preserved by subdivision.  Similarly, we would like our other invariants, such as $\dcat$, to be preserved by subdivision (cf.~\queref{que: dcat subdiv} above).
Further progress in this direction will likely involve extending the results of \cite{LOS19b} to higher-dimensional domains.

Second, and as indicated in the last paragraph of \secref{sec: cat and TC}, we intend to fully develop the notion of  Lusternik-Schnirelmann category in the digital setting and also include the notion of \emph{topological complexity} in a future treatment.  Topological complexity  is another numerical homotopy invariant that arises from the motion planning problem of topological robotics.  See \cite{Far06} for an introduction and \cite{G-L-V18} for some recent references.  With the results of \secref{sec: fibrations},
and especially \corref{cor: TC fibration}, we are already poised to embark on this.
These invariants could have implications for important problems in the digital setting,  such as feature recognition or image manipulation. 
For instance,  it was shown in \cite{GLO13} that in the ordinary topological setting, a space $X$ has topological complexity of one if and only if $X$ is an odd-dimensional sphere. In the  digital setting, it may be possible to use these invariants to recognize features such as circles or spheres.

Third, a broad goal of future work is to arrive at a characterization of a ``(higher-dimensional) digital sphere up to homotopy equivalence."     Towards this goal, we are adapting ideas from \cite{Evako2006, Evako2015}, which give a characterization of digital spheres that may be compared to homeomorphism. In the previous paragraph, we indicated how topological complexity might play a role in characterizing digital circles or spheres.  It is likely that more of the standard machinery of algebraic topology, such as homology and higher-dimensional homotopy groups, will need to be developed (in a way that incorporates our subdivision-oriented point of view) in order to progress in this direction.


\begin{thebibliography}{10}

\bibitem{ADFQ03}
R.~Ayala, E.~Dom\'{i}nguez, A.~R. Franc\'{e}s, and A.~Quintero, \emph{Homotopy
  in digital spaces}, Discrete Appl. Math. \textbf{125} (2003), no.~1, 3--24,
  9th International Conference on Discrete Geometry for Computer Imagery (DGCI
  2000) (Uppsala).

\bibitem{Bo-Ve18}
A.~Borat and T.~Vergili, \emph{Digital {L}usternik-{S}chnirelmann category},
  Turkish J. Math. \textbf{42} (2018), no.~4, 1845--1852. \MR{3843949}

\bibitem{Bo99}
L.~Boxer, \emph{A classical construction for the digital fundamental group}, J.
  Math. Imaging Vision \textbf{10} (1999), no.~1, 51--62. \MR{1692842}

\bibitem{Bo05}
\bysame, \emph{Properties of digital homotopy}, J. Math. Imaging Vision
  \textbf{22} (2005), no.~1, 19--26. \MR{2138582}

\bibitem{Bo06}
\bysame, \emph{Digital products, wedges, and covering spaces}, J. Math. Imaging
  Vision \textbf{25} (2006), no.~2, 159--171. \MR{2267137}

\bibitem{Bo18}
\bysame, \emph{Alternate product adjacencies in digital topology}, Appl. Gen.
  Topol. \textbf{19} (2018), no.~1, 21--53. \MR{3784715}

\bibitem{B-K10}
L.~Boxer and I.~Karaca, \emph{Some properties of digital covering spaces}, J.
  Math. Imaging Vision \textbf{37} (2010), no.~1, 17--26. \MR{2607636}

\bibitem{BK12}
\bysame, \emph{Fundamental groups for digital products}, Adv. Appl. Math. Sci.
  \textbf{11} (2012), no.~4, 161--179.

\bibitem{BS16}
L.~Boxer and P.~C. Staecker, \emph{Fundamental groups and {E}uler
  characteristics of sphere-like digital images}, Appl. Gen. Topol. \textbf{17}
  (2016), no.~2, 139--158.

\bibitem{B-S17}
L.~Boxer and P.~Christopher Staecker, \emph{Homotopy relations for digital
  images}, Note Mat. \textbf{37} (2017), no.~1, 99--126. \MR{3733806}

\bibitem{B-S18}
\bysame, \emph{Remarks on pointed digital homotopy}, Topology Proc. \textbf{51}
  (2018), 19--37. \MR{3633236}

\bibitem{CLOT}
O.~Cornea, G.~Lupton, J.~Oprea, and D.~Tanr{\'e},
  \emph{Lusternik-{S}chnirelmann category}, Mathematical Surveys and
  Monographs, vol. 103, American Mathematical Society, Providence, RI, 2003.

\bibitem{Eg-Ka17}
O.~Ege and I.~Karaca, \emph{Digital fibrations}, Proc. Nat. Acad. Sci. India
  Sect. A \textbf{87} (2017), no.~1, 109--114. \MR{3608877}

\bibitem{Evako2006}
A.~V. Evako, \emph{Topological properties of closed digital spaces: One method
  of constructing digital models of closed continuous surfaces by using
  covers}, Computer Vision and Image Understanding \textbf{102} (2006),
  134--144.

\bibitem{Evako2015}
\bysame, \emph{Classification of digital {$n$}-manifolds}, Discrete Appl. Math.
  \textbf{181} (2015), 289--296.

\bibitem{Far06}
M.~Farber, \emph{Topology of robot motion planning}, Morse theoretic methods in
  nonlinear analysis and in symplectic topology, NATO Sci. Ser. II Math. Phys.
  Chem., vol. 217, Springer, Dordrecht, 2006, pp.~185--230.

\bibitem{GLU14}
R.~Gonzalez-Diaz, J.~Lamar, and R.~Umble, \emph{Computing cup products in
  {$\Bbb{Z}_2$}-cohomology of 3{D} polyhedral complexes}, Found. Comput. Math.
  \textbf{14} (2014), no.~4, 721--744.

\bibitem{GR05}
R.~Gonz\'alez-D{\'\i}az and P.~Real, \emph{On the cohomology of 3{D} digital
  images}, Discrete Appl. Math. \textbf{147} (2005), no.~2-3, 245--263.
  \MR{2127077}

\bibitem{GLO13}
M.~Grant, G.~Lupton, and J.~Oprea, \emph{Spaces of topological complexity one},
  Homology Homotopy Appl. \textbf{15} (2013), no.~2, 73--81.

\bibitem{G-L-V18}
M.~Grant, G.~Lupton, and l.~Vandembroucq (eds.), \emph{Topological complexity
  and related topics}, Contemporary Mathematics, vol. 702, American
  Mathematical Society, Providence, RI, 2018. \MR{3762828}

\bibitem{H-M-P-S15}
J.~Haarmann, M.~P. Murphy, C.~S. Peters, and P.~C. Staecker, \emph{Homotopy
  equivalence in finite digital images}, J. Math. Imaging Vision \textbf{53}
  (2015), no.~3, 288--302. \MR{3397100}

\bibitem{Ha05}
S.~Han, \emph{Non-product property of the digital fundamental group}, Inform.
  Sci. \textbf{171} (2005), no.~1-3, 73--91. \MR{2123812}

\bibitem{Han10}
\bysame, \emph{K{D}-{$(k_0,k_1)$}-homotopy equivalence and its applications},
  J. Korean Math. Soc. \textbf{47} (2010), no.~5, 1031--1054.

\bibitem{Hir03}
Philip~S. Hirschhorn, \emph{Model categories and their localizations},
  Mathematical Surveys and Monographs, vol.~99, American Mathematical Society,
  Providence, RI, 2003. \MR{1944041}

\bibitem{KaIs18}
\.{I}smet Karaca and Melih \.{I}s, \emph{Digital topological complexity
  numbers}, Turkish J. Math. \textbf{42} (2018), no.~6, 3173--3181.
  \MR{3885444}

\bibitem{Kong89}
T.~Y. Kong, \emph{A digital fundamental group}, Computers and Graphics
  \textbf{13} (1989), 159--166.

\bibitem{Kong91}
T.~Y. Kong, R.~Kopperman, and P.~R. Meyer, \emph{A topological approach to
  digital topology}, Amer. Math. Monthly \textbf{98} (1991), no.~10, 901--917.

\bibitem{K-R-R92}
T.~Y. Kong, A.~W. Roscoe, and A.~Rosenfeld, \emph{Concepts of digital
  topology}, Topology Appl. \textbf{46} (1992), no.~3, 219--262, Special issue
  on digital topology. \MR{1198732}

\bibitem{LOS19c}
G.~Lupton, J.~Oprea, and N.~Scoville, \emph{A fundamental group for digital
  images}, Preprint, 2019.

\bibitem{LOS19b}
\bysame, \emph{Subdivision of maps in digital topology}, Preprint, 2019.

\bibitem{MVK14}
E.~T. Meric, T.~Vergili, and I.~Karaca, \emph{Computing higher dimensional
  digital homotopy groups}, Appl. Math. Inf. Sci. \textbf{8} (2014), no.~5,
  2417--2425.

\bibitem{Peters12}
J.~F. Peters and P.~Wasilewski, \emph{Tolerance spaces: origins, theoretical
  aspects and applications}, Inform. Sci. \textbf{195} (2012), 211--225.

\bibitem{Po71}
T.~Poston, \emph{Fuzzy geometry}, Thesis, University of Warwick, 1971.

\bibitem{Ro86}
A.~Rosenfeld, \emph{`{C}ontinuous' functions on digital pictures}, Pattern
  Recognition Letters \textbf{4} (1986), 177--184.

\bibitem{So86}
A.~B. Sossinsky, \emph{Tolerance space theory and some applications}, Acta
  Appl. Math. \textbf{5} (1986), no.~2, 137--167. \MR{823824}

\end{thebibliography}

\providecommand{\bysame}{\leavevmode\hbox to3em{\hrulefill}\thinspace}
\providecommand{\MR}{\relax\ifhmode\unskip\space\fi MR }
\providecommand{\MRhref}[2]{%
  \href{http://www.ams.org/mathscinet-getitem?mr=#1}{#2}
}
\providecommand{\href}[2]{#2}

\end{document}